\documentclass[11pt]{article}
\usepackage[utf8]{inputenc}
\usepackage{amsmath,amsthm}
\usepackage{amsfonts}
\usepackage{amssymb}
\usepackage{array}
\usepackage{bm}
\usepackage{multirow}
\usepackage{subfig}
\usepackage{enumitem}
\usepackage[normalem]{ulem}
\usepackage[font=small]{caption}

\usepackage[margin=2.9cm]{geometry}
\RequirePackage[colorlinks,citecolor=blue,urlcolor=blue]{hyperref}

\usepackage{algorithm}
\usepackage{algpseudocode}

\algnewcommand{\IIf}[1]{\State\algorithmicif\ #1\ \algorithmicthen}

\usepackage[	natbib=true,
	style=alphabetic,
	backend=bibtex,
	doi=false,
	isbn=false,
	url=false,
	firstinits=true,
	maxnames=4
]{biblatex}

\makeatletter
\def\blx@maxline{77}
\makeatother

\usepackage{arash_macros}
\usepackage{authblk}
\title{The neighborhood lattice for encoding partial correlations in a Hilbert space\thanks{This work was supported by NSF grant IIS-1546098.}}
\author{Arash A. Amini, Bryon Aragam, Qing Zhou}

\usepackage{arash_macros}

\providecommand{\Lc}{\mathcal{L}}

\providecommand{\Xc}{\mathcal{X}}
\providecommand{\Tc}{\mathcal{T}}

\providecommand{\Hil}{\mathcal{H}}
\DeclareMathOperator{\Span}{span}
\def\sep(#1,#2,#3){#1-#2-#3}

\providecommand{\Ss}{S^*}

\providecommand{\nhbdcoef}{\beta}

\providecommand{\nhbdset}{\mathcal{T}}

\providecommand{\dvec}{x}

\providecommand{\coef}{\beta}
\providecommand{\R}{\reals}

\providecommand{\tpose}{T}

\providecommand\independent{\protect\mathpalette{\protect\independenT}{\perp}}
\def\independenT#1#2{\mathrel{\rlap{$#1#2$}\mkern4mu{#1#2}}}

\def\given{\,\mid\,}

\providecommand{\norm}[1]{\Vert#1\Vert}

\definecolor{maroon}{rgb}{0.5, 0.0, 0.0}

\providecommand{\Cc}{\mathcal{C}}
\newcommand{\Pc}{\mathcal{P}}

\providecommand{\rv}{X}
\providecommand{\gv}{x}

\newcommand{\Fc}{\mathcal F}

\newcommand\Sigh{\widehat{\Sigma}}

\newcommand\Hc{\mathcal{H}}

\newcommand{\Kc}{\mathcal{K}}
\newcommand{\pa}{\Pi}

\newcommand{\Tcb}{\bm{\mathfrak{T}}} 

\begin{document}

	\maketitle
\begin{abstract}

	Neighborhood regression has been a successful approach in graphical and structural equation modeling, with applications to learning undirected and directed graphical models. We extend these ideas by defining and studying an algebraic structure called the neighborhood lattice based on a generalized notion of neighborhood regression. We show that this algebraic structure has the potential to provide an economic encoding of all conditional independence statements in a Gaussian distribution (or conditional uncorrelatedness in general), even in the cases where no graphical model exists that could ``perfectly'' encode all such statements. We study the computational complexity of computing these structures and show that under a sparsity assumption, they can be computed in polynomial time, even in the absence of the assumption of perfectness to a graph. On the other hand, assuming perfectness, we show how these neighborhood lattices may be ``graphically'' computed using the separation properties of the so-called partial correlation graph. 
    We also draw connections with directed acyclic graphical models and Bayesian networks.
    We derive these results using an abstract generalization of partial uncorrelatedness, called partial orthogonality, which allows us to use algebraic properties of projection operators on Hilbert spaces to significantly simplify and extend existing ideas and arguments. Consequently, our results apply to a wide range of random objects and data structures, such as random vectors, data matrices, and functions.
	
\end{abstract}

\section{Introduction}

Ascertaining the dependency structure of a system is a central task in machine learning and statistics. Accordingly, there are many tools for this including regression analysis, graphical models, factor analysis, and structural equation models (SEM). Unsurprisingly, these models are closely related \citep{wermuth1980,pourahmadi1999}, and the pros and cons of each approach are well-studied. At the same time, SEM-driven techniques have evolved as a useful tool for analyzing graphical models in high-dimensions \citep{geer2013,loh2014causal,aragam2016,drton2018algebraic}. As a result, understanding the relationship between the partial regression coefficients defined by a statistical model is an important problem in statistics. In this paper, we report some interesting properties of these partial regression coefficients and connect them to the properties of the underlying partial correlation graph (PCG). Our goals are twofold: (a) To introduce a new algebraic structure called the \emph{neighborhood lattice} for encoding dependency structures under minimal assumptions, and (b) In doing so, to highlight some new connections between SEM and graphical models. 

To motivate our results, consider the simple case of a Gaussian random vector $\rv=(\rv_{1},\ldots,\rv_{d})$. To understand the relationship between the variables in $\rv$, it is natural to inquire about the partial regression coefficients defined by regressing each $\rv_{j}$ onto the rest of the nodes $\rv_{-j}=\{\rv_{1},\ldots,\rv_{d}\} \setminus \{\rv_{j}\}$, the so-called \emph{neighborhood regression} problem \citep{meinshausen2006}. These coefficients encode a set of conditional independence (CI) relationships, namely, a zero coefficient implies that two nodes are conditionally independent given the rest of the variables. In this way the partial regression coefficients provide a convenient encoding of the dependency structure of the random vector $\rv$. Moreover, this structure is conveniently represented via the well-known Gaussian graphical model; see \cite{Lauritzen1996} for details. This procedure generalizes to various non-Gaussian models such as Ising models \citep{ravikumar2010}, exponential random families \citep{yang2015}, and semiparametric families \citep{Liu2009}. 

On the other hand, the (undirected) graphical model described above, in general, does not encode all the CI statements present in the underlying Gaussian distribution. The case where the distribution and the graph encode the same set of CI statements is known as the \emph{perfect} case. However, perfectness can be easily violated if the distribution is generated, say, from a true directed acyclic graph (DAG), sometimes referred to as a set of recursive SEMs. In these cases, there could be a DAG model perfect with respect to the distribution, in the sense of encoding exactly the same set of CI statements in the distribution, via the so-called $d$-separation in the DAG. Yet, there are still many distributions that are not covered by these two cases, i.e., not perfect with respect to any directed or undirected graphical model. 

In this work, we explore an alternative structure (i.e., other than a graph), which is defined by a generalization of classical neighborhood regression and can be used to economically encode \emph{all} of the CI statements in a Gaussian distribution, without assuming perfectness relative to a graph. 
More precisely, we wish to find all triplets $(j,k,T)$ such that $X_j\independent X_k \mid X_T$, where $T\subset \{1,\ldots,d\}\setminus\{j,k\}$. 
A complete list can be produced by regressing $\rv_{j}$ onto $\rv_{S}:= (\rv_k, k \in S)$ for general subsets $S\subset[d]_{j}:=\{1,\ldots,d\}\setminus\{j\}$, which we call \emph{generalized neighborhood regression}. 
Unfortunately, if we na\"ively consider all possible subsets, the problem explodes from a single neighborhood $[d]_{j}$ into $2^{d-1}$ neighborhoods. It is thus of interest to encode these {generalized neighborhood regression} problems more compactly.

One of the main contributions of this paper is to show that generalized neighborhood regression problems have a rich algebraic structure which we call the \emph{neighborhood lattice}, and that under certain natural sparsity assumptions, this structure in fact provides an economical encoding of \emph{all} of the CI relations in $\rv$. The neighborhood lattice is interesting in and of itself, both in how it provides an alternative to graphical models and in its intrinsic algebraic structure which has many useful properties. Among these useful properties, we will discuss potential computational savings when evaluating all possible conditional independence statements via the lattice (Section~\ref{sec:comp}) and also some connections with the complexity of learning DAG models (Section~\ref{sec:learning:directed}).

In the cases where the distribution is perfect to a conditional independence graph (CIG), we provide simple rules for graphical computation of the neighborhood lattices via the notion of graph separation in the CIG, thus showing how these algebraic representations connect with more familiar graphical models, when such models are possible. We develop all our results in the abstract setting of variables in a general Hilbert space $\Hil$, where the notion of independence is replaced with orthogonality, conditional independence is replaced with a corresponding notion of \emph{partial} (or \emph{conditional}) \emph{orthogonality}, and a CIG is replaced with a so-called partial correlation graph. Thus, our results extend beyond the Gaussian case as long as we replace conditional independence with the weaker notion of partial orthogonality. This connection between the axioms of conditional independence and more general operations such as Hilbert space projections is well-known \citep{pearl1988,Lauritzen1996,dawid2001separoids}, and our results are most naturally phrased in this language.

\subsection{Contributions of this work} 

At a high-level, we make the following specific contributions:
\begin{itemize}
\item We prove that a generalized notion of neighborhood regression has the structure of a complete, convex lattice (Theorem~\ref{thm:lattice}) that can be computed efficiently (Proposition~\ref{prop:lat:comp:complex}). 

\item The resulting lattice decomposition has a one-to-one correspondence with all the CI statements among $X_1\ldots,X_d$ (Theorem~\ref{thm:all:POs}). In the absence of perfectness,  we provide alternative conditions under which the lattice decomposition, hence all CI statements, can be computed in polynomial time (Theorem~\ref{thm:poly:N}).

\item Under the perfectness assumption, we provide two characterizations of these lattices in terms of the separation properties of the PCG (Theorem~\ref{thm:sep:char}) and its connected components (Theorem~\ref{thm:conn:comp}).
\end{itemize}
As far as we know, the lattice property of neighborhood regression has not been noted before in the literature. This useful property has interesting implications for the representation of CI relations (Section~\ref{sec:enumerate:POs}) as well as connections to the problem of learning DAG models from data (Section~\ref{sec:learning:directed}). As a bonus, we develop this theory in considerable generality, adopting the idea of replacing conditional independence with partial orthogonality from previous work in the graphical modeling literature. This also allows us to give largely self-contained proofs of all our results using operator-theoretic arguments. 

\subsection{Overview of the main results}

Let us give an overview of our results in the context of a motivating example.
Assume that we have a collection of random variables $X_1,\dots,X_d$ that are jointly Gaussian with covariance matrix $\Sigma \in \reals^{d \times d}$. Let $\beta_j(S)$ be the coefficient vector resulting from regressing $X_j$ onto $X_S$, which we refer to as SEM coefficients.  By varying $j$ and $S$, one can capture all CI relations in the \emph{sparsity pattern} (i.e. supports) of the resulting SEM coefficients $\beta_j(S)$. Moreover, one might be interested in these coefficients in their own right, not necessarily as a means to deriving CI relations.

Consider for example the inverse covariance matrix $\Gamma = \Sigma^{-1}$ shown in Figure~\ref{fig:perfect:pcg}(b). The sparsity pattern of $\Gamma$ can be thought of as the adjacency matrix of an undirected graph, the CIG associated with variables $X_1,\dots,X_d$ (Figure~\ref{fig:perfect:pcg}(c)). Consider for example, regressing $X_3$ on $\{X_{8},X_{9},X_{11},X_{12},X_{14}\}$:
\begin{align*}
	\beta_3(\{9,8,11,12,14\}) = \argmin_{\beta\, \in\, \reals^5} \ex[X_3 - (X_8,X_9,X_{11},X_{12},X_{14}) \beta ]^2.
\end{align*}
It turns out that in this regression, the coefficients of $X_8$ and $X_{11}$ are zero. That is, 
\begin{align*}
	\supp \big(\beta_3(\{9,8,11,12,14\})\big) = \{9,12,14\}.
\end{align*}
Which other variables besides those in $S = \{8,9,11,12,14\}$ can we include in this regression without the support changing? More generally, what are the other sets $S' \subset [d] \setminus \{j\}$ such that $\beta_j(S')$ has the same support as $\beta_j(S)$? Our main result (Theorem~\ref{thm:lattice}) shows that the set of such $S'$ has a rich algebraic structure, namely, it is a \emph{convex lattice} under the inclusion ordering. In particular, associated with every $j$ and $S$, there is a smallest and a largest set, $m_j(S)$ and $M_j(S)\subset [d] \setminus \{j\}$, respectively, such that regressing $X_j$ onto them  results in the same support for the regression coefficient vector. In this example, clearly $m_j(S) = \{9,12,14\}$, but the largest set is often nontrivial, and in this example it turns out to be $M_j(S) =  \{9,8,11,12,14,4,5,7,10,13,15\}$. A consequence of this result, for example, is that we can write
\begin{align}\label{eq:exa:reg}
	X_3 = \beta^*_9 X_{9} + \beta^*_{12} X_{12} + \beta^*_{14} X_{14} + \eps^*
\end{align}
where $\eps^*$ is independent of all the variables in $X_k, k \in M_j(S)$. Since $M_j(S)$ is the largest set by definition, there is no larger set with this property.

\begin{figure}
	\centering
	\begin{tabular}{ccc}
		\includegraphics[width=1.7in]{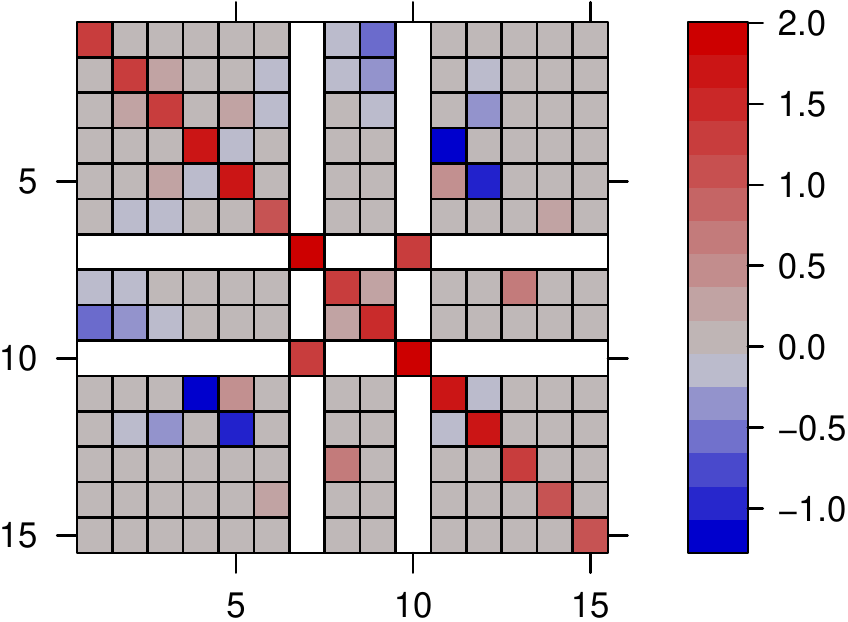} &
		\includegraphics[width=1.7in]{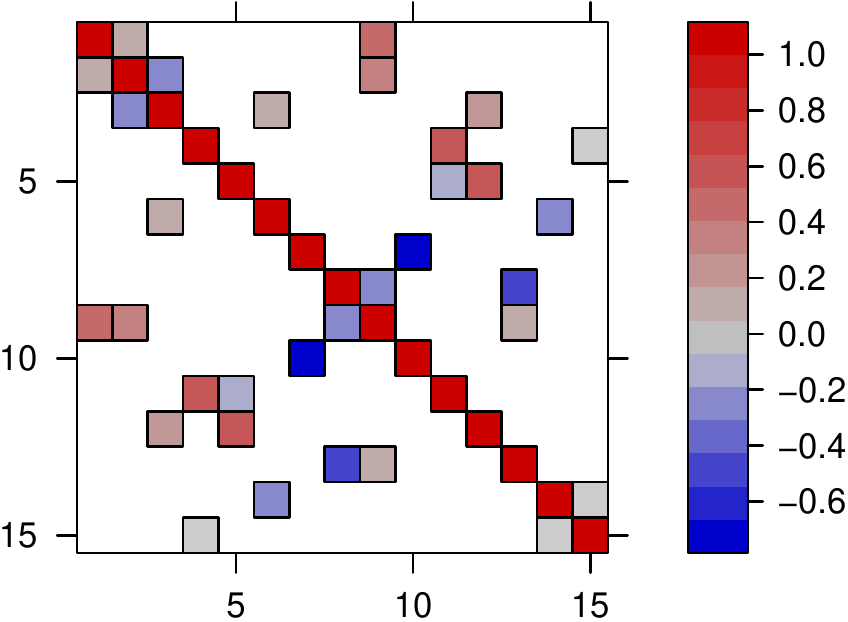} &
		\includegraphics[width=1.6in]{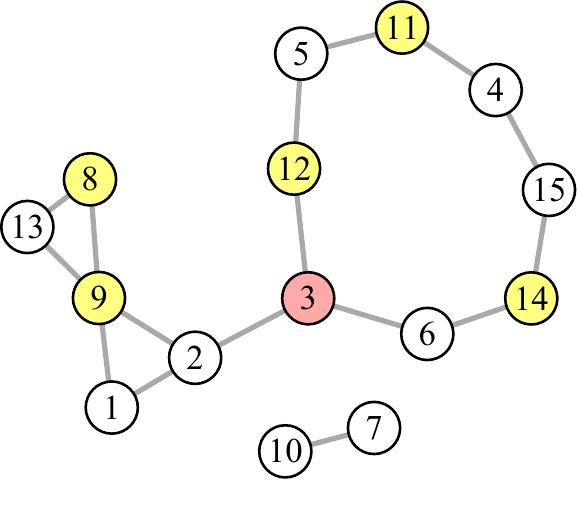} \\
		(a) & (b) & (c)
	\end{tabular}
	\caption{A Markov perfect example, with inverse covariance generated randomly: (a) covariance matrix,  (b) inverse covariance matrix, and (c) PCG. 
		In (c)  nodes are colored to help graphical computation of $\Tc_j(S)$ for $j=3$ and $S = \{8,9,11,12,14\}$.  See Examples~\ref{exa:perfect} and~\ref{exa:perfect:cont}.  
	}
	\label{fig:perfect:pcg}
\end{figure}

Any set $S'$ including $m_j(S)$ and contained in $M_j(S)$ has the property that $\supp(\beta_j(S)) = \supp(\beta_j(S'))$, and those are the only sets with this property. We denote this collection of sets as $\Tc_j(S)$ and write this in the interval notation
\begin{align}\label{eq:lat:examp}
	\begin{split}
	\Tc_j(S) &= [m_j(S),M_j(S)] \\
	&= \{S'\;:\; m_j(S) \subset S' \subset M_j(S)\} \\
	&= \big\{S'\;:\; \{9,12,14\} \subset S' \subset\{9,8,11,12,14,4,5,7,10,13,15\} \big\}.
	\end{split}
\end{align}
The regression coefficients in~\eqref{eq:exa:reg} will be the same for all subsets in $\Tc_j(S)$ and will be different for any subset outside $\Tc_j(S)$. The lattice $\Tc_j(S)$ compactly encodes many CI statements. In particular, we will show (Section~\ref{sec:enumerate:POs}) that lattice~\eqref{eq:lat:examp} encodes $X_i \independent X_j \mid X_{T}$ for any
\begin{align}
	i \in M_j(S) \setminus m_j(S) = \{8,11,4,5,7,10,13,15\}, \quad \text{and} \quad 
	T: \;  m_j(S) \subset T \subset M_j(S) \setminus \{i\}.  \end{align}
 Moreover, $\Tc_j(S)$ can be computed efficiently with $O(d)$ projections, a consequence of the convexity of the lattice (Section~\ref{sec:comp}).

We also observe that the neighborhood lattices $\Tc_j(S), S \subset [d] \setminus \{j\}$ provide a partition of the power set of $[d] \setminus \{j\}$, giving a complete picture of the relation of $X_j$ to other variables. We refer to this partition as the \emph{lattice decomposition} of node $j$ (Definition~\ref{defn:lat:decomp}) and show that calculating this decomposition provides us with all the CI statements involving node~$j$ (Section~\ref{sec:enumerate:POs}). We also provide conditions under which the decomposition can be computed in polynomial time (Section~\ref{sec:comp}).

The covariance matrix $\Sigma$ in this simple example was chosen to be perfect (see Section~\ref{sec:perfectness} for the definition), so that the reader can reason about the statements in the above discussion by inspecting the CIG. However, the lattice property of $\Tc_j(S)$ (i.e., Theorems~\ref{thm:lattice} and~\ref{thm:all:POs}) is true in general, without any perfectness assumption. It is known that without perfectness, one cannot infer all CI statements in the joint distribution from graph separation. Therefore, the neighborhood lattice is a more general representation for CI statements than a CIG. If $\Sigma$ is indeed perfect, we develop simple graphical rules to compute $\Tc_j(S)$ from the CIG associated with $\Gamma =\Sigma^{-1}$ (Theorems~\ref{thm:sep:char} and~\ref{thm:conn:comp}).

All these ideas carry out in the more general setting of abstract PCGs, with conditional independence replaced with the notion of \emph{partial orthogonality (PO)}. Although this general setting has appeared in the literature before, we briefly give an overview in Sections~\ref{sec:partial:orth}, \ref{sec:PCG:proj} and~\ref{sec:perfectness} using the language of projection operators which simplifies many arguments. For completeness, other aspects of the theory of abstract PO and PCGs will be developed in Appendix~\ref{app:detials}. 

\subsection{Related work}

Early work relating the properties of partial regression coefficients in a structural equation model to the independence properties of a Gaussian model include \citet{wright1921,wright1934,dempster1969,cramer1946}.
The general problem of encoding conditional independence (and more generally, partial uncorrelatedness) has a long history. For a historical summary and early work on this problem, we refer the reader to textbooks such as \cite{pearl1988} and \cite{Lauritzen1996}. Graphical models are the most common encoding, whereby conditional independence is encoded via separation properties in a graph. These properties can be axiomatized via the semi-graphoid and graphoid axioms, which has led to a rich literature on the axiomatic foundations of conditional independence; see \cite{pearl1988,studeny2006probabilistic} and the references therein. More recently, the literature has considered various extensions of traditional undirected and directed graphs to encode more general independence structures \citep[e.g.][]{evans2014markovian,heckerman2014variations,richardson2017nested,sadeghi2014markov,lauritzen2018unifying}. To the best of our knowledge, the lattice property of partial uncorrelatedness has not yet been explored, although the connection between orthogonal projections and the axioms of conditional independence is of course well-known \citep{Lauritzen1996,dawid2001separoids,whittaker2009graphical}.

\section{Preliminaries}

In this section, we review some preliminaries and background material that is useful for setting the stage of our general theory. The material on projections in Section~\ref{sec:proj:lattice} in particular is essential for the definitions and results in the sequel.

\subsection{Notation}
\paragraph{Graph notation.}
In an undirected graphical model for a random vector $\rv = (\rv_1,\dots,\rv_d)$, we identify each random variable $\rv_{j}$, $j=1,\ldots,d$, with a node in an undirected graph $G=([d],E)$, where $[d] = \{1,\dots,d\}$. 
Two nodes $i$ and $j$ are \emph{adjacent}, or \emph{neighbors}, if $(i,j)\in E$, in which case we write $i \sim j$, otherwise $i \nsim j$. A \emph{path} from $i$ to $j$ is a sequence $i = k_1,k_2,\ldots,k_{n-1},k_{n}=j \in[d]$ of distinct elements with $(k_{l},k_{l+1})\in E$ for each $l=1,\ldots,n-1$. 
Given two subsets $A,B\subset [d]$, a path connecting $A$ to $B$ is any path with $k_{1} \in A$ and $k_{n}\in B$.
A subset $C\subset [d]$ \emph{separates} $A$ from $B$, denoted by  $A-C-B$, if all paths connecting $A$ to $B$ intersect $C$ (i.e. $k_{l}\in C$ for some $1<l<n$), otherwise we write $ \neg (\sep(A,C,B))$. Implicit in this definition is that  $A,B$ and $C$ are disjoint.

\paragraph{Subset notation.} For any $A\subset [d]$, we denote $\rv_{A} = \{\rv_i:\; i \in A\}$. We also use the shorthand notations: $\{i\} = i$ and $\{i,j\} = ij$, $A \cup \{i\} = Ai$, $A \cup B = AB$ and so on.
In addition, we let $[d]_S = [d] \setminus S = \{1,\dots,d\} \setminus S$. Common uses of these notational conventions are: $[d]_j = [d] \setminus \{j\}$ and $[d]_{ij} = [d] \setminus ij = [d] \setminus \{i,j\}$.
For a matrix $\Sigma \in \reals^{d \times d}$, and subsets $A,B \subset [d]$, we use $\Sigma_{A,B}$ for the submatrix on rows and columns indexed by $A$ and $B$, respectively. Single index notation is used for principal submatrices, so that $\Sigma_{A} = \Sigma_{A,\,A}$. For example, $\Sigma_{i,j}$ is the $(i,j)$th element of $\Sigma$ (using the singleton notation), whereas $\Sigma_{ij} = \Sigma_{ij,\,ij}$ is the $2\times 2$ submatrix on $\{i,j\}$ and $\{i,j\}$. 
Similarly, $\Sigma_{Ai,Bj}$ is the submatrix indexed by rows $A \cup \{i\}$ and columns $B \cup \{j\}$.

\subsection{Graphical models}
\label{subsec:gm}
For completeness, we recall here some of the traditional definitions from the graphical modeling literature. This is mainly to provide a basis for comparison, and is not needed for most of the results we prove. A more thorough treatment can be found in \cite{koller2009,Lauritzen1996}. 

In the context of undirected graphs, there are three so-called \emph{Markov properties} of interest: $G$ satisfies the
\begin{itemize}
\item \emph{pairwise Markov property} if $(i,j)\notin E$ implies that $\rv_{i}\independent \rv_{j}\given \rv_{[d]_{ij}}$,
\item \emph{local Markov property} if each node is conditionally independent of all other variables given its neighbours,
\item \emph{global Markov property} if $\sep(A,C,B)$ implies that $\rv_{A}\independent \rv_{B}\given \rv_{C}$.
\end{itemize}
The notation $\rv_1 \independent \rv_2 \mid \rv_3$ means $\rv_1$ is (probabilistically) independent of $\rv_2$ conditioned on $\rv_3$. In general, these conditions are not equivalent and we have global$\implies$local$\implies$pairwise, however, for positive distributions, these conditions are in fact equivalent~\citep{Lauritzen1996}. 
A graph $G$ satisfying the pairwise Markov property is sometimes called a \emph{conditional independence graph} for the distribution of $\rv$.

The Markov properties go in one direction: The graphical criterion (i.e. separation) implies a probabilistic conclusion (i.e. conditional independence). The converse is not true in general; when it is, the graph $G$ is called \emph{perfect} with respect to the joint distribution of $\rv$. Specifically, $G$ is called \emph{perfect} if $\sep(A,C,B) \iff \rv_{A}\independent \rv_{B}\given  \rv_{C}.$ 

\subsection{The lattice of projections}
\label{sec:proj:lattice}

We will assume the reader is familiar with the basic theory of Hilbert spaces and their projections; a detailed introduction to these topics can be found in \cite{Farah2010,Blackadar2006}. For a (separable) Hilbert space $\Hil$, let $B(\Hil)$ be the space of bounded linear operators on $\Hil$. For an operator $P \in B(\Hil)$, let $\ran(P) := P \Hil := \{P x :\; x \in \Hil\}$ denote its range and $P^*$ its adjoint, defined via the relation $\ip{x,P^*y} = \ip{Px,y}$ for all $x,y \in \Hil$. Here and in the sequel, we use $\ip{\cdot,\cdot}_\Hil$ or the shorthand $\ip{\cdot,\cdot}$ to denote the inner product of $\Hil$. 

 We recall that $P \in B(\Hil)$ is an orthogonal projection if and only if it is self-adjoint and idempotent: $P^* = P = P^2$. The set of orthogonal projections in $B(\Hil)$, denoted as $\Pc(\Hil)$, is of particular importance in this paper. $\Pc(\Hil)$ can be partially ordered as follows:
\begin{lem}
    For $P,Q \in \Pc(\Hil)$, the following are equivalent:
    \begin{align*}
        (a) \; P = PQ, \quad (b)\; P = QP, \quad (c)\;\ran(P) \subset \ran(Q).
    \end{align*}
\end{lem}
\noindent    When any of these conditions hold we write $P \le Q$.

The equivalence of (a) and (b) follows from self-adjointness of orthogonal projections. Note that when $P$ and $Q$ are both projections, $PQ$ is not necessarily a projection (unless $P$ and $Q$ commute).   The above ordering makes $\Pc(\Hil)$ into a complete lattice, that is, a partially ordered set $\mathcal{S}$ in which each subset has both a supremum (join) and an infimum (meet) in $\mathcal{S}$. For $P,Q \in \Pc(\Hil)$, we denote the supremum with $P \vee Q$ and the infimum with $P \wedge Q$.
We have
\begin{align*}
    P \wedge Q &=  \text{the projection onto $\ran(P) \cap \ran(Q)$} \\
    P \vee Q &= \text{the projection onto the closed linear span of $\ran(P) \cup \ran(Q)$.}
\end{align*}
We also let $P^\perp  := I - P$, the orthogonal complement of $P$. Note that $P \vee P^\perp = I$ (the identity operator) and $P \wedge P^\perp = \{0\}$. Also, $P \le Q$ iff $P^\perp \ge Q^\perp$. See for example~\citet[Section 5, p.~24]{Farah2010} and~\citet[Section~II.3.2, p.~78]{Blackadar2006} and the references therein.

We often consider the concrete example where $\Hil = L^2(\pr)$, the space of random variables (on a probability space with measure $\pr$) with finite second moments.  When applied to a random variable $X$, $P^{\perp}X=X-PX$ may be interpreted as the residual between $X$ and its projection $PX$. For this reason, when $x$ is a vector in an arbitrary Hilbert space $\Hil$ we will refer to $P^{\perp}x$ as the \emph{residual} of the projection $Px$.
We also need the notion of orthogonality of operators:
\begin{lem}\label{lem:op:orth}
    For $P,Q \in B(\Hil)$, the following are equivalent:
    \begin{center}
        (a) $\ip{Px,Qy} = 0,\; \forall x,y\in \Hil$, \quad (b) $P^*Q = Q^* P = 0$, 
        \quad (c) $\ran(P) \perp \ran(Q)$.
    \end{center}
\end{lem}
\noindent
When any of the conditions above hold, we call $P$ and $Q$ \emph{orthogonal} and denote this by $P \perp Q$.

\section{General results on neighborhood lattices}

In contrast to classical probability theory---which works with random variables and the notions of independence and correlation---we will develop our results in the abstract setting of a general Hilbert space $\Hil$. 
Since this extension might be unfamiliar to readers, in Section~\ref{sec:partial:orth} we give an outline of this theory. More details on this framework can be found in Appendix~\ref{app:detials}. 
In particular, the notion of partial correlation among variables is replaced with that of \emph{partial orthogonality} (cf. \eqref{eq:partial:orth})---which is well-defined in any Hilbert space--- and will play a prominent role in the rest of the paper. Later in section, we state our main results on neighborhood lattices and how they encode all partial orthogonality statements.

\subsection{Partial orthogonality}
\label{sec:partial:orth}

We start with a generalized definition of partial uncorrelatedness (or conditional independence in the case of Gaussian variables). Throughout, we write $\Hil^d = \{(\gv_1,\dots,\gv_d):\; \gv_i \in\Hil, \, i \in [d]\}$ for some positive integer\footnote{We assume $d < \infty$ for simplicity, though most of our results hold also for $d = \infty$.} $d$ and a separable Hilbert space $\Hil$. 
Fix a vector $\gv = (\gv_1,\dots,\gv_d) \in \Hil^d$, and for any $S \subset [d]$, let $P_S \in \Pc(\Hil)$ denote the projection onto the span of $\gv_S = \{\gv_i,\; i\in S\}$. For simplicity, we write $P_j = P_{\{j\}}$ to denote projection on the span of $\{x_j\}$, i.e., we drop the brackets for singleton sets.

\begin{defn}[Partial orthogonality]\label{defn:PO}
	For any disjoint triple of subsets $A,S,B \subset [d]$, we say that variables $x_A$ are partially orthogonal to $x_B$ given $x_S$ if 
 	\begin{align}\label{eq:partial:orth}
 		P_A P_S^\perp P_B = 0,
 	\end{align}
 	in which case we also write $P_A \perp P_B \mid P_S$ or $(A \perp B \mid S)$.
\end{defn}

We will also view~\eqref{eq:partial:orth} as a ternary operation among projection operators and say that ``$P_A$ is partially orthogonal to $P_B$ given $P_S$''. Since  $P_A P_S^\perp P_B = P_B P_S^\perp P_A$ due to self-adjointness of projections, Definition~\ref{defn:PO}  is symmetric in $A$ and $B$.  
Note the use of $\perp$ here to denote orthogonality in $\Hil$, as opposed to $\independent$ which is reserved for probabilistic independence. In particular, recall that $P_S^\perp := I - P_S$ is the orthogonal complement of $P_S$ (Section~\ref{sec:proj:lattice}). We refer to~\eqref{eq:partial:orth} as \emph{partial orthogonality}, or \emph{conditional orthogonality}, a generalization of the notion of partial uncorrelatedness of random variables. For $\Hil = L^2(\pr)$, we recover the usual notion for random variables:
\begin{exa}\label{exa:usual:rvs}
    For a probability measure $\pr$, let $\Hil = L^2(\pr)$, the space of square-integrable random variables: $X_i \in L^2(\pr)$ iff $\ex X_i^2 < \infty$. The inner product is $\ip{\rv, Y}_{L^2} = \ex XY$. Consider zero-mean random variables $\rv_1,\dots,\rv_d \in L^2(\pr)$. In this case, $P_S$ corresponds to the $L^2$ projection on the span of $\rv_{S}$, or \emph{regressing on} (covariates) $\rv_i, i \in S$ in statistical terms. Similarly, $P_S^\perp \rv_i$ is the \emph{residual} after regressing $\rv_i$ on $\rv_S$. 
    
    Now, fix some $S \subset [d]$. The usual partial correlation between $\rv_i$ and $\rv_j$ given $\rv_S$ is the correlation between the residuals $P_S^\perp \rv_i$ and $P_S^\perp \rv_j$, that is, $\ip{P_S^\perp \rv_i,P_S^\perp \rv_j}_{L^2}$. Thus, the partial correlation is zero iff $P_S^\perp \rv_i \perp P_S^\perp \rv_j$ (orthogonality in the $L^2$ sense). Since $\Span\{P_S^\perp \rv_j\} = \ran(P_S^\perp P_j)$, using Lemma~\ref{lem:op:orth}(c), the partial correlation is zero iff $P_S^\perp P_i \perp P_S^\perp P_j $. This in turn is equivalent to $P_i P_S^\perp P_j = 0$, by Lemma~\ref{lem:op:orth}(b). Thus, for random variables in $L^2$, the notion of partial orthogonality given in Definition~\ref{defn:PO} coincides with that of partial uncorrelatedness.
     \qed
\end{exa}

Let $\Sigma = ( \ip{\gv_i,\gv_j}_\Hil) \in \reals^{d \times d}$ be the \emph{Gram matrix} of the underlying variables $\gv = (\gv_j,\; j\in [d])$. Unless otherwise stated, we will assume the following from now on:
\begin{align}\label{asum:PSD:Sigma}
\Sigma := (\ip{\gv_i,\gv_j}_\Hil)  \succ 0. \tag{A1}
\end{align} 
The ternary relation \eqref{eq:partial:orth} among projections is also completely characterized by the Gram matrix $\Sigma$ as discussed in Appendix~\ref{sec:parorth:gram}. 

\subsection{Neighborhood regression}
\label{sec:nbhd:reg}

It is well-known that neighborhood regression provides an alternative characterization of the partial orthogonality relation~\eqref{eq:partial:orth}.  As will be shown in Section~\ref{sec:proof:enumerate:POs}, we can verify the relation $P_i \perp P_j \mid P_S$, by looking at the coefficients of the regression of $x_j$ onto $\{x_k,\; k \in Si\}$, as opposed to the residuals. 

In the literature on undirected graphical models, the neighborhood of a fixed node $\dvec_{j}$ is implicitly understood to be the set of all other variables, namely the collection $\dvec_{-j}$. Here, we extend the notion of neighborhood to any subset $S \subset [d]_j$, which is necessary for evaluating general conditional orthogonality statements, such as $x_i \perp x_j \mid x_S$. (See Section~\ref{sec:directed} for a discussion of how this general setup is needed for learning directed extensions of the PCG.) The \emph{neighborhood regression problem} for the node $\dvec_{j}$ and a subset $S\subset[d]_{j}$ is the problem of finding the linear projection of $\dvec_{j}$ onto $\dvec_{S}$:
\begin{defn}[SEM coefficients]
    \label{defn:nhbd}
    For any $S \subset [d]_j$, let 
    \begin{align}
    \label{eq:def:nhbd:regression}
    \coef_j(S) &:= \argmin_{\beta \,\in\, \R^d, \;\supp(\beta)\, \subset\, S}
    \norm{\dvec_{j} - \beta^T\dvec}_{\Hil}^{2}.
    \end{align}
    We call $\coef_j(S)$ the SEM coefficients for variable $j$ regressed on the variables $S$.  Here, $\beta^T x = \sum_{j=1}^d \beta_j x_j$ for any $\beta \in \reals^d$ and $x \in \Hil^d$.
\end{defn}

In Section~\ref{sec:proof:enumerate:POs}, we will show that to test $P_i \perp P_j \mid  P_S$ for all $S \subset [d]_{ij}$, it is enough to have $\{\beta_j(T):\, T \subset [d]_j\}$, motivating the study of such collections  (cf. Proposition~\ref{prop:proj:Sigma:gen}) .   It turns out that these neighborhood regression problems have rich algebraic properties which will be explored in the next section. 

\subsection{Neighborhood lattice}\label{sec:nbhd:lattice}
We now state our main results. We start by introducing the the main object of study in the paper, which we call the \emph{neighborhood lattice}.

\begin{defn}[Neighborhood lattice]
    \label{defn:nhbd:lattice}
    For any $S \subset [d]_j$, define a collection of subsets by
    \begin{align}\label{eq:lattice:proj:def}
    \Tc_j(S) = \{ T \subset [d]_j:\; P_T P_j = P_S P_j\}.
    \end{align}
    
\end{defn}
\noindent
In other words, for any $j$, $\nhbdset_j(S)$ is the collection of candidate sets $T\subset[d]_{j}$ such that the projection of $\dvec_{j}$ onto $\{\dvec_i, \, i\in T\}$ is invariant.  Definition~\ref{defn:nhbd:lattice} is somewhat abstract, though it highlights the algebraic nature of the collection $\Tc_j(S)$ and the sufficiency of the projection operators $\{P_S, S \subset [d]\}$ in defining it. Note that $\Tc_j(S)$ as given in~\eqref{eq:lattice:proj:def} is well-defined even when the Gram matrix $\Sigma$ is rank deficient.   In order to see the usefulness of this definition, assume in addition that $\Sigma$ is (strictly) positive definite. Then we have the following alternative representations of $\Tc_j(S)$ (Lemma~\ref{lem:equiv:lattice:defs} in Section~\ref{sec:details:nbhd:lat}):
\begin{align}
\begin{split}
\nhbdset_j(S) &= \{T \subset[d]_j: \; \nhbdcoef_j(T) = \nhbdcoef_j(S) \} \\
&= \big\{T \subset[d]_j: \; \supp(\nhbdcoef_j(T)) = \supp(\nhbdcoef_j(S)) \big\}.
\label{eq:lattice:alt}
\end{split}
\end{align}
By definition, $ \supp(\beta_j(S)) \subset S$. We are primarily interested in cases where $\supp(\beta_j(S))$ is much smaller than $S$. In addition, we would like to know all the other sets $T$ for which $\supp(\beta_j(T))$ is the same as this small set. 

Our goal in this paper is to study the algebraic properties of $\nhbdset_j(S)$, for which representation~\eqref{eq:lattice:proj:def} is very helpful. Throughout, we fix $j \in [d]$ and $S \subset [d]_j$. Our next result shows that $\nhbdset_j(S)$ is indeed a lattice, as suggested by the name. Recall that a complete lattice is a partially ordered set, or \emph{poset} for short, in which all subsets have both a supremum (join) and an infimum (meet)~\cite[Section 3.3]{stanley1997enumerative}.  We also need the following definition: Let $(\Pc, \le)$ be a poset and $\Cc$ a subposet of $\Pc$. We say that $\Cc$ is \emph{convex} (in $\Pc$) if $z \in \Cc$ whenever $x < z < y$ with $x,y \in \Cc$. A closed interval $[x,y] := \{z \in \Pc:\; x \le z \le y\}$ is an example of a convex subposet~\cite[Section 3.1]{stanley1997enumerative}.

\begin{thm}\label{thm:lattice}
	Under~\eqref{asum:PSD:Sigma},
    $\Tc_j(S)$, ordered by inclusion, is a complete lattice. In particular, $\nhbdset_j(S)$ has unique minimal and maximal elements, which we denote by $m_j(S)$ and $M_j(S)$, respectively. Moreover,  $\Tc_j(S)$ is convex as a subposet of $2^{[d]_j}$. \end{thm}
Any finite nonempty lattice is in fact complete, however, the proof we give in Section~\ref{sec:proof:thm:lattice} also establishes the result for an infinite collection of variables $\{x_i\}$. 
 $\Tc_j(S)$ being convex in $2^{[d]_j}$ means the following: For every two subsets $S_1,S_2 \in \Tc_j(S)$, and any $S' \subset [d]_j$, if $S_1 \subset S' \subset S_2$, then $S' \in \Tc_j(S)$. Thus, using the interval notation, we can represent the neighborhood lattice~as 
 \begin{align}\label{eq:lattice:interval:rep}
    \Tc_j(S) = [m_j(S), M_j(S)]  =\{ A \in 2^{[d]_j} :\; m_j(S) \subset A \subset M_j(S)\}.
 \end{align}
  In other words, $\Tc_j(S)$ is specified by its minimum and maximum elements, and it is actually isomorphic to the lattice of the power set $2^{[r]}$ with $r=|M_j| - |m_j|$. 
  In Section~\ref{sec:graphical:comp}, we describe the abstract sets $m_j(S)$ and $M_j(S)$ in terms of the PCG $G$ (Theorems~\ref{thm:sep:char} and~\ref{thm:conn:comp}), which gives a useful semantic interpretation in terms of separation properties of a graph.  Theorem~\ref{thm:lattice} is in fact a special case of a much more abstract result discussed in Appendix~\ref{sec:abs:lattice}.

\begin{exa}\label{exa:perfect}    Consider the covariance matrix $\Sigma$ shown in Figure~\ref{fig:perfect:pcg}(a), with dimension $d = 15$.   The corresponding PCG is plotted in Figure~\ref{fig:perfect:pcg}(c). 
    For $j = 3$ and $S = \{9,8,11,12,14\}$, the minimal set is $m_j(S) = \{9,12,14\}$, and maximal set is 
    $M_j(S) = \{4,5,7,8,9,10,11,12,\allowbreak13,14,15\}$. These can be easily read from the graph as will be described in Section~\ref{sec:graphical:comp}. Since the lattice can be represented as $\Tc_j(S) = [m_j(S), M_j(S)]$, we know for example that $\{9,12,14,5\}$ and $\{9,12,14,10,13\}$ both belong to $\Tc_j(S)$. Note that these two sets are not comparable in $2^{[d]_j}$, i.e., $\Tc_j(S)$ is not a \emph{chain} (or totally ordered set) in $2^{[d]_j}$. In this case, $\Tc_j(S)$ contains $2^{|M_j| - |m_j|} = 2^{8} = 256$ subsets, out of a total possible $|2^{[15]_3}| = 2^{14}$. Regressing $j=3$ onto any of these 256 subsets results in the same SEM coefficients.
\end{exa}

An alternative way to view $\Tc_j(S)$ is as a collection of subsets of $\{x_i:\; i \in [d]_j\}$  that have the same explaining power as $S$ for $x_j$. This view can be further advanced by removing the reference to the set $S$. A node $j$ defines an equivalence relation among subsets of $[d]_j$, where $S,T \subset [d]_j$ are equivalent iff $P_S P_j = P_T P_j$. We then say that $S$ and $T$ have the same explaining power for $\{j\}$. This relation partitions $2^{[d]_j}$ into equivalence classes, each of which is a complete lattice according to Theorem~\ref{thm:lattice}. For future reference, let us give this partition a name:
\begin{defn}[Lattice decomposition]\label{defn:lat:decomp}
	For any $j \in [d]$, we call the partition of $2^{[d]_j}$ into neighborhood lattices, the (neighborhood) \emph{lattice decomposition} of node $j$, denoted as $\Tcb_j$.  More precisely, each element $\Tc_j \in \Tcb_j$ is an interval $\Tc_j = [m_j,M_j]$ as in Theorem~\ref{thm:lattice}, and we have $\bigcup \Tcb_j = 2^{[d]_j}$ and $\Tc_j \cap \Tc_j' = \emptyset$ for distinct $\Tc_j, \Tc_j' \in \Tcb_j$.
\end{defn}

\begin{rem}
	\label{rem:activeset}
    It is not hard to see that $m_j(S)=\supp(\nhbdcoef_j(S))$, that is, $m_j(S)$ represents the so-called \emph{active set} when regressing $\gv_{j}$ onto $\gv_{S}$, which is of course of interest in regression and graphical models. Thus, there is a one-to-one correspondence between the neighourbohood lattices $\Tc_{j}(S)$, the equivalence classes defined above, and the possible active sets for predicting $x_{j}$. Furthermore, a consequence of Theorem~\ref{thm:lattice} is that there will be a largest set $M_j(S)$ possibly larger than the $S$ we started with, regressing onto which has the same active set. Regressing onto any set strictly larger than $M_j(S)$ will change the active set.
\end{rem}

\subsection{Enumerating all partial orthogonality statements}\label{sec:enumerate:POs}

An immediate application of Theorem~\ref{thm:lattice} is that it provides an economical way of enumerating all partial orthogonality (PO) statements that hold among variables $\dvec_i, \, i\in [d]$ without any perfectness assumption. It is enough to focus on statements of the form $P_i \perp P_j \mid P_T$ for any $T \subset [d]_{ij}$ which we represent compactly as $(i \perp j \mid T)$, implicitly assuming that $T$ does not include $\{i,j\}$. We further fix $j$, i.e., restrict attention to all PO statements involving a particular node $j$.  As mentioned earlier, when the variables are prefect w.r.t. to a graphical model (either undirected or directed), inferring the underlying PCG is often a good economical way of representing all PO statements. However, without assuming perfectness, one has to verify each statement $(i \perp j \mid T)$ for all $i \neq j$ and all $T \subset [d]_{ij}$. Theorem~\ref{thm:lattice} shows that once we identify a neighborhood lattice $\Tc_j$, we can obtain many of these statements for free. More precisely, we have:
\begin{prop}\label{prop:all:POs:for:Tj}
	Let $\Tc_j = [m_j,M_j]$ be a neighborhood lattice for node $j$. Assume that $T$ and $\{i\}$ are disjoint subsets of $[d]_j$, and  $T \cup\{i\} \in \Tc_j$. Then $(j \perp i \mid  T)$ if and only if $i \in M_j \setminus m_j$.
\end{prop}
\noindent
See Section~\ref{sec:proof:enumerate:POs} for the proofs of the results in this section. 

 Proposition~\ref{prop:all:POs:for:Tj} allows us to enumerate all PO statements assuming that we have the full lattice decomposition $\Tcb_j$ (Definition~\ref{defn:lat:decomp}): 
 \begin{thm}\label{thm:all:POs}
 	Let $\Tcb_j=\big\{[m_j^\ell,M_j^\ell]:\ell=1,\ldots,K_j\big\}$. Then $j \perp i \mid  T$ if and only if
 	there exists a unique $k\in [K_j]$ such that $$i \in M^k_j \setminus m^k_j, \quad T \in [m_j^k,M_j^k \setminus \{i\}].$$
 \end{thm}
 
In short, once we have the lattice decomposition $\Tcb_j$ for node $j$, then all the PO statements involving node $j$ can be listed as follows:
\begin{align}\label{eq:all:POs:listed}
	j \perp i \mid  T \quad \text{for all} \quad 
	i \in M^k_j \setminus m^k_j, \quad T \in [m^k_j,M^k_j \setminus \{i\}], \quad k=1,\dots,K_j. \end{align}
The total number of PO statements in~\eqref{eq:all:POs:listed} is
\begin{align}\label{eq:POs:count}
	\sum_{k=1}^{K_j} \big(|M_j^k| - |m_j^k|\big)\, 2^{|M_j^k|- |m_j^k| - 1}, 
\end{align}
where, for example, $|M_j^k|$ is the cardinality of set $M_j^k$. (Note that the total number of all possible PO triples involving a fixed node $j$ is $(d-1) 2^{d-2}$.)
According to Theorem~\ref{thm:all:POs}, there is no double-counting in~\eqref{eq:all:POs:listed} and all the PO statements (involving $j$) are accounted for. Thus, the complexity of enumerating all PO statements, in the absence of perfectness assumption, boils down to the complexity of computing decomposition~$\Tcb_j$. In Section~\ref{sec:comp}, we show that computing an individual lattice $\Tc_j$ has  polynomial complexity and we also give sufficient conditions under which computing $\Tcb_j$ could be done in polynomial time. Thus, in the absence of a graphical representation, the lattice decomposition provides an alternative ``economical encoding'' of PO statements with potential for substantial savings.

\subsection{Computational complexity}\label{sec:comp}

\newcommand{\Scand}[0]{candidates\xspace}
\newcommand{\minimals}[0]{minimals\xspace}
\let\oldReturn\Return
\renewcommand{\Return}{\State\oldReturn}

\begin{algorithm}[t]
	\begin{algorithmic}[1]
		\Function{computeLattice}{$j$, $S$, $\Sigma$} 
		\State $\beta_j(S) \gets$ \textproc{computeSEMcoeff}($j$, $S$, $\Sigma$)
		\State $m \gets \supp(\beta_j(S))$. \Comment{This is $m_j(S)$.}
		\State $M \gets S$
		\State $A \gets [d] \setminus (S\cup\{j\})$
		\While{ A $\neq \emptyset$ } 
		\State $k \gets $ (pop first element off $A$)
		\State $P \gets M \cup \{k\}$ \Comment{Proposal set}
		\State $c \gets$ \textproc{computeSEMcoeff}($j$, $P$, $\Sigma$)
		\IIf{$\supp(c) = m$} $M \gets P$    		\EndWhile           
		\Return{$[m,M]$}
		\EndFunction
		
		\Function{computeSEMcoeff}{$j$,$S$,$\Sigma$}        
		\Return $\beta_j(S)$  which is $= \Sigma_S^{-1} \Sigma_{S,j}$ on $S$ and $0$ otherwise.
		\EndFunction
	\end{algorithmic}
	\caption{Compute the lattice $\Tc_j(S)$ for a given $j$ and $S \subset [d]_j$.\label{alg:lattice:comp}}
\end{algorithm}

Theorem~\ref{thm:lattice}, combined with the observations in Remark~\ref{rem:activeset}, suggests an intuitive algorithm for efficiently computing the entire neighborhood lattice $\Tc_j(S)$. Once we have the active set $m_j(S)$, it is enough to compute $M_j(S)$ by sequentially adding the rest of the nodes and either including or rejecting them based on whether the active set is changed. The validity of this approach follows from the convexity of the lattice. Once $m_j(S)$ and $M_j(S)$ are computed, the entire lattice is given by~\eqref{eq:lattice:interval:rep}. This procedure is outlined formally in Algorithm~\ref{alg:lattice:comp}.
Since this algorithm only requires $d -1 - |m_j(S)| = O(d)$ projections, we have the following perhaps surprising result:
\begin{prop}\label{prop:lat:comp:complex}
For any $j$ and $S\subset[d]_{j}$, it is possible to compute $\Tc_j(S)$ with $O(d)$ projections. 
\end{prop}
Since each projection can be computed in polynomial time (i.e. $O(d^3)$), a consequence of Proposition~\ref{prop:lat:comp:complex} is that  $\Tc_j(S)$ can be computed in polynomial time.

Now consider the question of computing the full lattice decomposition $\Tcb_j$ which is needed in general to enumerate all PO statements by Theorem~\ref{thm:all:POs}. This can be done recursively, as detailed in Algorithm~\ref{alg:lattice:decomp}. The key in this algorithm is step~\ref{step:finding:uncovered:S} where given a partial decomposition, say, $\Tcb_j' = \{\Tc_j^1,\dots,\Tc_j^K\}$, one needs to find a set $S$ that is not covered by $\Tcb_j'$, i.e., $S \notin \bigcup \Tcb_j' = \bigcup_{\ell=1}^K \Tc_j^\ell$. If no such set exists, we conclude that $\Tcb_j'$ covers the power set (i.e., $\bigcup \Tcb_j' = 2^{[d]_j}$), hence it is in fact the full lattice decomposition and the algorithm terminates. For a general set system $\Tcb_j'$, performing step~\ref{step:finding:uncovered:S} is NP-hard. However, since in our case the underlying lattices are mutually disjoint, the problem can be solved in polynomial time:
\begin{lem}\label{lem:decide:set:system}
	Given a set system $\Tcb_j' =\{\Tc_j^1,\dots,\Tc_j^K\}$ where $\Tc_j^\ell = [m_j^\ell,M_j^\ell] \subset 2^{[d]_j}$ for all $\ell=1,\dots,K$ and $\Tc_j^\ell \cap \Tc_j^{\ell'} = \emptyset$ for $\ell \neq \ell'$, there is a polynomial-time algorithm to decide whether $\Tcb_j'$ covers the power set $2^{[d]_j}$ and if not produce a certificate (i.e., a set not covered by $\Tcb_j'$). The algorithm requires at most $O( d^2 K)$ set operations.
\end{lem}

\begin{algorithm}[t]
	\begin{algorithmic}[1]
		\State Set $\Tc_j^1 \gets \text{computeLattice}(j, [d]_j,\Sigma)$
		\State Initialize $\Tcb_j \gets \{\Tc_j^1\}$, $K \gets 1$ and $\text{powerSetNotCovered} \gets \text{True}$.
		\While{ powerSetNotCovered }
		\State Find a set $S \subset [d]_j$ not covered by $\Tcb_j = \{\Tc_j^1,\dots,\Tc_j^K\}$, that is, $S \notin \bigcup_{\ell=1}^K \Tc_j^\ell$. \label{step:finding:uncovered:S}
		\If {no such set $S$ exists}
		\State $\text{powerSetNotCovered} \gets \text{False}$.
		\Else
		\State $\Tc_j^{K+1} \gets \text{computeLattice}(j, S,\Sigma)$.
		\State $\Tcb_j \gets \{\Tc_j^1,\dots,\Tc_j^K, \Tc_j^{K+1}\}$.
		\State $K \gets K+1$.
		\EndIf
		\EndWhile

	\end{algorithmic}
	\caption{Compute the lattice decomposition $\Tcb_j$ for a given $j$~\label{alg:lattice:decomp}.}
\end{algorithm}

A set operation in the statement of Lemma~\ref{lem:decide:set:system} (and Theorem~\ref{thm:poly:N} below) involves the computation of the size of at most two set differences. We refer to the proof in Section~\ref{sec:proof:comp} for more details.
Since the maximum $K$ achieved in Algorithm~\ref{alg:lattice:decomp} could  potentially be much smaller than $2^{d-1}$, the overall procedure could result in substantial savings relative to the naive approach of checking all the $2^{d-1}$ possible subsets. In general, we have the following result for the complexity of Algorithm~\ref{alg:lattice:decomp}, and by proxy that of enumerating all PO statements:

\begin{thm}\label{thm:poly:N}
	Assume that lattice decomposition $\Tcb_j$ contains $K_j$ lattices. Then, $\Tcb_j$ can be computed using at most $O(K_j(d^2K_j \;\text{set}  + d \;\text{projection}))$ operations. In particular, if $|m_{j}(S)|\le k$ for all $S\subset[d]_{j}$, then all partial orthogonality statements for the $j$th node can be computed in polynomial time, with at most $O(d^{2k+2} \,\text{set}  + d^{k+1} \,\text{projection})$ operations. 
\end{thm}

\noindent
In other words, as long as no active set for $x_{j}$ has more than $k$ elements, the computation of the lattice decomposition is polynomial in $d$. The proof of Theorem~\ref{thm:poly:N} follows by combining Proposition~\ref{prop:lat:comp:complex} and Lemma~\ref{lem:decide:set:system}, and noting that under the \emph{bounded active set} assumption,  the total number of lattices is bounded as $K_j \le\binom{d-1}{k}\le d^{k}$.
Whether there are other assumptions, besides bounded active set, that could guarantee a polynomial number $K_j$ of lattices is an interesting open question.

For the interested reader, \texttt{R} code implementing Algorithms~\ref{alg:lattice:comp} and~\ref{alg:lattice:decomp}, as well as the graphical computation of Section~\ref{sec:graphical:comp}, is available at~\cite{reg_lattice_comp}. 

\begin{exa}
	Continuing with the example covariance matrix $\Sigma$ shown in Figure~\ref{fig:perfect:pcg}(a), we apply Algorithm~\ref{alg:lattice:decomp} to compute the lattice decomposition $\Tcb_j$ for node $j=3$. The decomposition turns out to have $K_j=319$ lattices. Tables~\ref{tab:covered:dist} and~\ref{tab:mcount:dist} contain some statistics about this decomposition, namely, how many sets are covered by each lattice and the sizes of their minimum element (i.e., set $m$). For example, Table~\ref{tab:covered:dist}, shows that there are 5 lattices in the decomposition that cover $512$ sets each. Similarly, there are $21$ lattices whose minimal element is a set of size $5$. In this example, the maximum active set size is $k=5$. We note that the bound $\binom{d-1}{k} = \binom{14}{5} = 2002$ on the number $K_j$ of lattices  is quite conservative ($K_j = 319 \ll 2002$) which is attributed to the fact that many lattices in the decomposition have minimal elements of size smaller than $k$. The total number of (valid) PO statements for $j=3$, as given by~\eqref{eq:POs:count}, is 62,592. This is the number of all PO statements involving node $j=3$ that hold in this model, out of the $14 \cdot 2^{13} = 114,688$ possible such PO statements.
\end{exa}

\begin{table}[t]
	\centering
	\begin{tabular}{lrrrrrrrrrr|r}
		& & & & & & & & & & & Total \\
		\hline
		Number of sets covered & 4 & 8 & 16 & 32 & 64 & 128 & 256 & 512 & 1024 & 2048 & $16384$\\ 
		\hline
		Number of lattices & 112 &  40 &  60 &  34 &  38 &  19 &   8 &   5 &   2 &   1 & $319$\\ 
		\hline
	\end{tabular}
	\caption{The distribution of the number of sets covered by each lattice $\Tc = [m,M]$, i.e.  $2^{|M|-|m|}$, for the setup of Example~\ref{exa:perfect}.}\label{tab:covered:dist}
\end{table}
\begin{table}[t]
	\centering
	\begin{tabular}{rrrrrrr}
		\hline
		size of $m$ & 0 & 1 & 2 & 3 & 4 & 5 \\ 
		\hline
		number of lattices &   1 &  12 &  60 & 134 &  91 &  21 \\ 
		\hline
	\end{tabular}
	\caption{The distribution of the size of the minimal element of the lattice $\Tc = [m,M]$, i.e. set $m$, for the setup of Example~\ref{exa:perfect}.}\label{tab:mcount:dist}
\end{table}

\section{Graphical computation with PCG}\label{sec:graphical:comp}

So far we have developed properties of the neighborhood lattice and its computation,
without assuming the existence of a perfect graph $G$ for the Gram matrix $\Sigma$. 
In this general situation, the lattice decomposition encodes all partial orthogonality statements.
When a perfect PCG $G$ exists, all PO statements can be read off from graph separation in $G$.
Then it is expected that one can characterize the lattices from the graph $G$. In this section, we will show how to
compute a neighborhood lattice from a perfect PCG. 
We begin by generalizing the usual notion of a PCG to the Hilbert space setting, and then extend the Markov perfectness defined in Section~\ref{subsec:gm} based on this generalization.

\subsection{Abstract partial correlation graphs}\label{sec:PCG:proj}

For any $S \subset [d]$ let $P_S \in \Pc(\Hil)$ denote the projection onto the span of $\gv_S = \{\gv_i,\; i\in S\}$ and recall that $[d]_{ij} = \{1,\dots,d\} \setminus \{i,j\}$. 

\begin{defn}[Pairwise $\Hil$-Markov, PCG]\label{def:pairwise:L2:abs:PCG}
	\label{defn:pcg}
	We say that $\gv$ satisfies the \emph{pairwise $\Hil$-Markov property} w.r.t. $G$ if
	\begin{align}\label{eq:pairwise:L2}
	\text{$i \nsim j$ in $G$} \iff P_i P_S^\perp P_j = 0, \;\text{for}\; S = [d]_{ij},
	\end{align}
	in which case $G$ is called a \emph{partial correlation graph} of $\gv$.
\end{defn}

Recalling Example~\ref{exa:usual:rvs} (Section~\ref{sec:partial:orth}), where $\Hil = L^2(\pr)$ and we are dealing with random variables $X_1,\dots,X_d \in L^2(\pr)$,  Definition~\ref{defn:pcg} defines the PCG as the graph where a missing edge between a pair of nodes $(i,j)$ corresponds to zero partial correlation between $\rv_i$ and $\rv_j$ given the rest of the variables $\rv_k, k \in [d]_{ij}$. This is the usual notion of the PCG for a collection of random variables. Definition~\ref{defn:pcg} thus generalizes this notion to represent partial orthogonality (Definition~\ref{defn:PO}) in a Hilbert space.

In Example~\ref{exa:usual:rvs}, it is well-known that if $\rv =(\rv_1,\dots,\rv_d)$ has a Gaussian distribution, then the PCG defined above is in fact a CIG. However, the PCG is defined for any (joint) distribution on $\rv$ whose marginals have finite second moments, and in general it may not correspond to a CIG. (It is usually sparser than a CIG.) In Appendix~\ref{sec:pcg:examples}, we give more concrete examples of the abstract PCG of Definition~\ref{def:pairwise:L2:abs:PCG} which go beyond the familiar case of Example~\ref{exa:usual:rvs}.

\subsection{$\Hil$-Markov Perfectness}\label{sec:perfectness}
Let us now extend Definition~\ref{defn:pcg}. In analogy with the global Markov property in the context of CIGs, we can define a global notion of the $\Hil$-Markov property with respect to a graph: 
\begin{defn}[Global $\Hil$-Markov]\label{defn:global:Markov:L2}	Given the setup of Definition~\ref{defn:pcg}, we say that $\gv\in\Hil^d$ satisfies the global $\Hil$-Markov property w.r.t. $G$ if 
	\begin{align}\label{eq:defn:global:Markov:L2}
	\text{ $S$ separates $A$ and $B$ in $G$} \implies P_A  P_S^\perp P_B = 0.
	\end{align} 
\end{defn}

\noindent
In the special case $\Hil = L^2(\pr)$ (cf.~Example~\ref{exa:usual:rvs}), this reduces to the global Markov property for an undirected graphical model, and hence Definition~\ref{defn:global:Markov:L2} naturally generalizes this property to general Hilbert spaces.

We now introduce the notion of $\Hil$-Markov perfectness, which is the PCG counterpart of the notion of Markov perfectness in undirected graphical models. The definition of the global $\Hil$-Markov property in Definition~\ref{defn:global:Markov:L2} requires that graph separation implies partial orthogonality. Perfectness upgrades this implication to hold both ways.

\begin{defn}[Perfectness]\label{defn:perfectness}	We say that $\gv$ is globally $\Hil$-Markov perfect w.r.t. $G$ if 
	\begin{align*}
	\text{ $S$ separates $A$ and $B$ in $G$} \iff P_A  P_S^\perp P_B = 0.
	\end{align*}    
\end{defn}
\noindent
We often abbreviate ``globally $\Hil$-Markov perfect w.r.t. $G$'' and just refer to $\gv$ as perfect or imperfect. 

Markov perfectness is useful when making statements regarding the separation of nodes in the PCG, in which case there is a one-to-one correspondence between separation on the graph and the relation $P_A \perp P_B \mid P_S$ among projections. 

\begin{rem}[Perfectness of $\Sigma$]
	There is no need to specify $G$  in Definition~\ref{defn:perfectness}, since it is implied by $x$, due to the uniqueness of PCGs according to Definition~\ref{def:pairwise:L2:abs:PCG}. 	 Since PCGs and  $\Hil$-Markov properties can be equivalently characterized by the Gram matrix $\Sigma = (\ip{x_i,x_j}_\Hil)$ as discussed in Appendix~\ref{sec:parorth:gram}, we can equivalently talk about perfectness of a Gram matrix $\Sigma$. The corresponding graph is also (uniquely) implied in this case, given by the sparsity pattern of $\Sigma^{-1}$; see~\eqref{eq:inv:Gram} and Lemma~\ref{lem:Gram:char:pair}. 
\end{rem}

\subsection{Graphical characterization}
Under the perfectness assumption (Definition~\ref{defn:perfectness}), we can characterize the elements of $\Tc_j(S)$ in terms of the separation in the underlying PCG. For any set $S\subset[d]_{j}$, define a new set by
\begin{align}
\label{eq:Sstar}
\Ss:=\bigcap \{T \subset S:\; T \; \text{separates $j$ and $S \setminus T$}\}.
\end{align}
Thus, $\Ss$ is the smallest subset of $S$ that separates $j$ and $S \setminus \Ss$.  See Lemma~\ref{lem:sep:1} (Section~\ref{sec:seperation}) for the validity of this interpretation. Our next result gives the following characterization of the minimal set $m_j(S)$, the maximal set $M_j(S)$, and the entire lattice, via graph separation:
\begin{thm}\label{thm:sep:char} Assume $\Hil$-Markov perfectness.
    Fix $j \in [d]$, $S \subset [d]_j$, and define $S^{*}$ by \eqref{eq:Sstar}.
    Let $E_j(\Ss) = \{k:\; \text{$\Ss$ separates $k$ and $j$}\}$. Then, 
    \begin{itemize}
        \item[(a)] $m_j(S) = \Ss$.
        \item[(b)] $M_j(S) = \Ss \cup E_j(\Ss)$.
        \item[(c)] $\Tc_j(S) = \{\Ss \cup T:\; T \subset E_j(\Ss)\}$.
    \end{itemize}
\end{thm}

We next provide a characterization of the minimal and maximal sets of the neighborhood lattice of $j$ via connected components of the graph resulting from the removal of $j$:

\begin{thm}\label{thm:conn:comp} 
\sloppypar{Under $\Hil$-Markov perfectness, suppose that removing node $j$ and the edges connected to it breaks the PCG into $K$ connected components given by the vertex subsets $G_1,G_2,\dots,G_K \subset [d]_j$. Let $S_k = S \cap G_k$. Then,}
    \begin{itemize}
        \item[(a)] $m_j(S) = \biguplus_k m_j(S_k)$,
        \item[(b)] $M_j(S) = \bigcup_k M_j(S_k) = \biguplus_k M_j(S_k;G_k)$, where $M_j(S_k;G_k)$ is the largest element of
        \begin{align*}
            \Tc_j(S_k;G_k) :=  \{ T \subset G_k :\; P_T P_j = P_{S_k} P_j\}.
        \end{align*}
    \end{itemize}
\end{thm}
    \noindent
    In these statements $\biguplus$ denotes the \emph{disjoint} union.
    Note that $\Tc_j(S_k;G_k)$ is the lattice restricted to the ground set $G_k \cup \{j\}$ (i.e. instead of $[d]$). The original lattice given in Definition~\ref{defn:nhbd:lattice} can be thought of as $\Tc_j(S_k;[d]_j)$.
	For illustrative purposes, some simple consequences of Theorems~\ref{thm:sep:char} and~\ref{thm:conn:comp} are as follows:
    \begin{itemize}
        \item[(i)] $m_j(S) = \emptyset$ iff there is no path between $j$ and $S$. (Separation by the empty set.)
        \item[(ii)] If the PCG decomposes into two disjoint components, say $G$ and $H$, then $M_j(S)$ contains $H$ for any $j\in G$ and any $S$.    \end{itemize}
	Both (i) and (ii) hold without the perfectness assumption, and can be shown directly. However, the graphical view of Theorems~\ref{thm:sep:char} and~\ref{thm:conn:comp} makes them immediately clear.
 
\begin{figure}
    \centering
    \begin{tabular}{ccc}
        \includegraphics[width=1.8in]{figs/pcg_colored.pdf} \quad\quad &
        \includegraphics[width=2in]{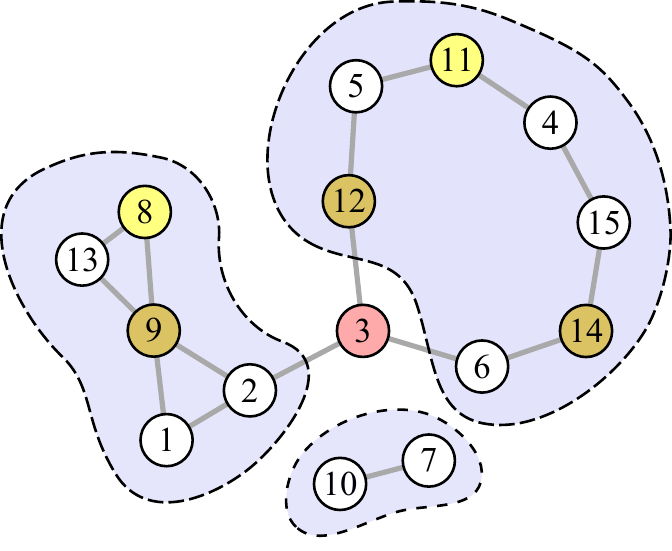}  \\
        (a) & (b)     \end{tabular}
    \caption{The graphical computation in Example~\ref{exa:perfect:cont} of the neighborhood lattice $\Tc_j(S)$ for $j=3$ and $S = \{9,8,11,12,14\}$. (a) Original PCG with $j$ and $S$ specified with different colors. 
        From this figure, it is easy to see that $\Ss = \{9,12,14\}$ is the smallest subset of $S$ that separates $j$ from $S \setminus \Ss$, hence $m_j(S) = \Ss$. 
                (b)~Illustration of the three connected components which result after removing node $j=3$. As a result of Theorem~\ref{thm:conn:comp}, we can work in each component separately. For example, restricted to the left component $G_1 = \{1,2,8,9,13\}$, the minimal subset of $S \cap G_1 = \{8,9\}$ that separates $j=3$ from the rest of $S \cap G_1$ is $\Ss_1 = \{9\}$. Within $G_1$, $\Ss_1$ separates $\{8,13\}$ from $j=3$. Hence, the restricted lattice $\Tc_j(S\cap G_1;G_1) = [\{9\},\{9,8,13\}]$.}
    \label{fig:pcg:graph:comp}
\end{figure}

\begin{exa}\label{exa:perfect:cont}
    Continuing with Example~\ref{exa:perfect}, let $j=3$ and  $S =\{9,8,11,12,14\}$. It is clear from Figure~\ref{fig:pcg:graph:comp}(a) that the smallest subset $\Ss$ of $S$ separating $j$ from $S \setminus \Ss$ is $\{9,12,14\}$. Hence, $m_j(S) = \Ss$ according to Theorem~\ref{thm:sep:char}. Similarly, $E_j(\Ss) = \{4,5,7,8,10,11,13,15\}$ and $M_j(S) = \Ss \cup E_j(\Ss)$.
    To verify Theorem~\ref{thm:conn:comp}, note that removing $j$ breaks the PCG into three connected components $G_1 = \{1,2,8,9,13\}$, $G_2 = \{7,10\}$ and $G_3 = \{4,5,6,11,12,14,15\}$, as illustrated in Figure~\ref{fig:pcg:graph:comp}(b). Applying Theorem~\ref{thm:sep:char}, we can compute the restricted lattices $\Tc_j(S \cap G_k; G_k), k =1,2,3$. Using the notation $[m,M]$ to represent a lattice with minimum and maximum elements $m$ and $M$, respectively, the three lattices are: $[\{9\},\{8,9,13\}]$, $[\{12,14\},\{4,5,11,12,14,15\}]$ and $[\emptyset, \{7,10\}]$. It is clear that the minimal and maximal elements of the original lattice are the disjoint union of the corresponding elements of these three lattices.
\end{exa}

As  Example~\ref{exa:perfect:cont} shows, graphical computation (assuming perfectness holds) could be more  efficient than direct computation. The PCG can be constructed from the sparsity pattern of the inverse Gram matrix, $\Sigma^{-1}$, after which all the lattice computations reduce to establishing certain connectivity criteria on the graph. 

\section{Directed models}\label{sec:directed}

Just as the undirected PCG defined in Definition~\ref{defn:pcg} is the $L^{2}$ analogue of undirected conditional independence graphs, it is possible to extend the general $L^{2}$ concepts to \emph{directed graphs} in a way that mirrors the theory of Bayesian networks (BNs) and structural equation models. In this section, we discuss such extensions and illustrate how these ideas are closely related to neighborhood regression and the neighborhood lattice introduced in Section~\ref{sec:nbhd:lattice}. 

\subsection{Directed PCGs}\label{sec:directed:pcg}

There are several equivalent ways to define the notion of a \emph{directed PCG}, which we describe here. 
Let $G$ be a DAG on $[d]$, and let $\pa_j$ denote the parent set of node $j$ in $G$. Similarly, let $N_j$ be the set of non-descendants of $j$ in $G$ (i.e., $i \in N_j$ if no directed path exists from $j$ to $i$), which is a well-defined notion due to the acyclicity assumption. 

\begin{defn}
    We say that $\gv = (\gv_1,\dots,\gv_d) \in \Hil^d$ satisfies a \emph{directed PCG} w.r.t. $G$~if 
    \begin{align}\label{eq:dPCG:def1}
    P_i P_{\pa_j}^\perp P_j = 0, \quad \forall j \in [d],\; i \in N_j.
    \end{align}
\end{defn}
This condition is modeled after the local Markov properties of BNs. Intuitively,~\eqref{eq:dPCG:def1} says that the residual after projecting $x_j$ onto its parents is orthogonal to any of $x_j$'s non-descendants. 
Alternatively, we could require that the residuals after projecting $x_i$ and $x_j$ onto their respective parent sets are orthogonal, which gives a condition that is symmetric in $i$ and~$j$:
\begin{align}\label{eq:dPCG:def2}
    P_i P_{\pa_i}^\perp P_{\pa_j}^\perp P_j = 0, \quad \text{for all distinct}\; i,j \in [d].
\end{align}
A common way to rewrite these conditions is to require that there exists a vector $e \in \Hil^d$, and a matrix $B \in \reals^{d \times d}$ with zero diagonal entries, whose associated graph is $G$ (that is, $\{(k,j) : B_{kj}\ne 0\}$ is the edge set of $G$) such that
\begin{align}
\label{eq:sem} (a)\; x = B^T x  + e,\quad 
(b) \; (B^T x)_i \perp e_i, \quad 
(c) \;e_i \perp e_j, \quad \forall i,j \in [d].
\end{align}
Here, $B^T x$ is interpreted as an element of $\Hil^d$ with entries $(B^T x)_j = \sum_{k} B_{kj} x_k$. The model \eqref{eq:sem} is often called a  (recursive) \emph{structural equation model} for $x$. 

We have opted to define a directed PCG via \eqref{eq:dPCG:def1} owing to its familiarity from the literature on directed graphical models. The following lemma, however, establishes the equivalence of conditions (\ref{eq:dPCG:def1}-\ref{eq:sem}):
\begin{lem}
Let $\gv = (\gv_1,\dots,\gv_d) \in \Hil^d$ be a vector satisfying~\eqref{asum:PSD:Sigma} and $G$ a directed acyclic graph. Then conditions \eqref{eq:dPCG:def1}, \eqref{eq:dPCG:def2}, and \eqref{eq:sem} are equivalent. 
\end{lem}

\noindent
The equivalence of~\eqref{eq:dPCG:def1} and~\eqref{eq:dPCG:def2} is proved in Appendix~\ref{sec:dPCG:equiv}.
The equivalence of~\eqref{eq:dPCG:def2} and~\eqref{eq:sem} is established via the following correspondences: $(B^T x)_j = P_{\pa_j} x_j$ and $e_j = P_{\pa_j}^\perp x_j$. We note that $B$ will not be unique unless the Gram matrix $\Sigma = (\ip{x_i,x_j})$ is nonsingular. Assuming that $\Sigma$ is nonsingular,  the $j$th column of $B$ is  $\beta_j(\pa_j)$ as given by Definition~\ref{defn:nhbd}. 

\newcommand{\xt}{\widetilde{x}}
\newcommand{\Bt}{\widetilde{B}}
It is always possible to obtain a directed PCG for a given $x \in \Hil^d$ by fixing an ordering of the elements of $x$ and performing \emph{recursive projection}, i.e., projecting each element onto those that come before it in the ordering. More precisely, fix a $d \times d$ permutation matrix $P$, and let $\xt= P x$. Then, we set $e_1 = \xt_1$ and proceed by projecting $\xt_j$ onto $\xt_1,\dots,\xt_{j-1}$ for $j \ge 2$, and calling the residual $e_j$. It is easy to see that this procedure leads to an SEM of the form $\xt = \Bt^T \xt + e$ where $\Bt^T$ is lower triangular. In terms of the original vector $x$, we obtain an SEM of the form~\eqref{eq:sem} with $B = B(P) = P^T \Bt P$, hence a directed PCG as defined above.

 In the other direction any directed PCG is obtained in this way, i.e. if~\eqref{eq:sem} holds, one can obtain a permutation matrix $P$ such that $P B P^T$ is upper triangular (due to acyclicity assumption). Furthermore, letting $\pi_P: [d] \to [d]$ be the permutation associated with $P$ and defining $S_{j}=S_{j}(P)=\pi_P^{-1}(\{1,\dots,j\})$, the $j$th column of the matrix $B$ obtained in this way corresponds to the coefficients $\beta_{j}(S_{j})$ as  given by Definition~\ref{defn:nhbd}. More interestingly, the support of $\beta_j(S_j)$ could be smaller than $S_j$ (the original candidate parent set used in the recursive projection). This support will be the minimal element of the corresponding neighborhood lattice, $\Tc_j(S_j)$, and will act as the (actual) parent set of $j$ in the constructed PCG: that is, $\Pi_j = m_j(S_j)$ and $\beta_j(S_j) = \beta_j(\Pi_j)$. The recursive projection procedure thus 
establishes an interesting connection between undirected PCGs, directed PCGs, SEM coefficients, and the neighborhood lattice which can be exploited when learning graphs from data (Section~\ref{sec:learning:directed}).

From the above discussion, it is clear that for any $x \in \Hil^d$, there are in fact many potential directed PCGs: One for each possible ordering of its elements (note that some orderings might lead to the same directed PCG). The corresponding SEM coefficients $B$ can be related to the Cholesky factors of the inverse Gram matrix, after proper permutation: Assume that~\eqref{eq:sem} holds, $P$ is a permutation matrix that makes $P B P^T$ upper triangular (always possible for a DAG), and $D$ is the diagonal Gram matrix of $e$. Letting $\Sigma$ be the Gram matrix of $x$, it is not hard to see that 
\[
\Sigma^{-1} = (I-B) D^{-1} (I-B)^T.
\]
Performing an ``upper'' Cholesky decomposition on $P \Sigma^{-1} P^T = U U^T$ (i.e., $U$ is an upper triangular matrix) we obtain $ U = P(I-B) D^{-1/2} P^T$ due to the uniqueness of the Cholesky factorization for positive definite matrices. From this interpretation, it is clear that the SEM in \eqref{eq:sem}---alternatively a directed PCG---is not unique, i.e., for every permutation $P$ of the elements of $x$, there is a corresponding SEM obtained from the upper Cholesky decomposition of $P \Sigma^{-1} P^T$. This distinguishes directed PCGs from undirected PCGs, which are always unique.

\subsection{Learning directed PCGs}\label{sec:learning:directed}

The SEM model \eqref{eq:sem} provides a useful, interpretable way to describe how the variables in $x=(x_{1},\ldots,x_{d})$ relate to one another that is commonly used in applications such as biology, social science, and machine learning. Given its apparent utility in practice, it is of significant interest to learn the coefficient matrix $B$ from data. Furthermore, since the matrix $B$ can be interpreted as a weighted adjacency matrix for the underlying DAG $G$, learning $B$ also implies learning the structural relations encoded by the graph $G$.

Unfortunately, since directed PCGs are not unique these models are unidentifiable. In the previous subsection we illustrated how the notion of a directed PCG is naturally related to an implicit ordering on the variables, represented by a permutation $P$. To circumvent identifiability issues, one often seeks permutations that result in the sparsest DAGs, as these represent parsimonious explanations of the variables that are simple to interpret in practice. If we know the order of the variables that leads to the sparsest possible directed graph $G$, then we can use the recursive projection procedure outlined in Section~\ref{sec:directed:pcg} to estimate $G$. More specifically, we can use \eqref{eq:def:nhbd:regression} with the neighborhoods $S_{j}$ defined previously to estimate the support of each $\beta_{j}(S_{j})$ and hence the graph $G$.

In practice, of course, we do not know this ordering. Thus, we must consider all possible $d!$ orderings of the variables. A naive algorithm would compute $\beta_{j}(S)$ for all possible neighborhoods $S\subset[d]_{j}$ and all $j\in[d]$, and check all possible permutations to find the sparsest graph $G$. This is clearly computationally infeasible. In considering all possible neighborhoods, however, one implicitly encounters the neighborhood lattices $\Tc_j(S)$ (Definition~\ref{defn:nhbd:lattice}) for every possible $S \subset [d]_j$ and $j \in [d]$. These lattices encode how much ``redundancy'' exists in the different neighborhoods, and suggests that we can reduce the total number of neighborhoods one must consider by exploiting the algebraic structure of these lattices. This leads one to wonder: Is it possible to exploit this redundancy in order to learn directed PCGs more efficiently?

Without sparsity or other simplifying assumptions, there are $2^{d-1} d$ neighborhood problems to consider, which is intractable when $d$ is large. By imposing a natural sparsity assumption on the neighborhood lattices, however, the total number of neighborhood problems reduces substantially owing to the redundancies encoded by the neighborhood lattices. For example, if we assume that $|m_j(S)| \le k \ll d$, for all $S \subset [d]_j$ and $j \in [d]$, the total number of neighborhood problems reduces to at most $\binom{d}{k} d\le d^{k+1} \ll 2^{d-1} d$. This is an immediate consequence of the lattice construction in Section~\ref{sec:nbhd:lattice}. By Remark~\ref{rem:activeset}, we have $m_{j}(S)=\supp(\beta_{j}(S))$, so that this assumption amounts to assuming that the parent sets are sparse, which is a natural assumption to make in practice. This suggests that one can learn all possible SEM representations of the variables using a much smaller sample size, relative to the naive approach of solving all possible regressions. This has been successfully exploited for Gaussian models \citep{aragam2016} to achieve optimal sample complexity $n=\Omega(k\log d)$ in learning recursive structural equation models \citep{ghoshal2017limits}. The results presented in the current work suggest that these ideas can be extended much further to \emph{non-Gaussian} models, a direction we intend to pursue in the future.

\section{Discussion}\label{sec:discuss}
{In this paper, we investigated the algebraic structure of generalized neighborhood regression, motivated by the problem of enumerating all partial orthogonality relations efficiently under minimal assumptions. Our results on the neighborhood lattice are presented with a generalized notion of partial orthogonality in a Hilbert space as the counterpart of conditional independence. We further explored PCGs, as a generalization to conditional independence graphs, and their connections to neighborhood lattices.} Finally, we discussed some subtleties and complications in extending these ideas to directed models. Studying these extensions in more depth is left for future work. Below, we discuss an additional avenue for future work regarding the estimation of partial orthogonality relations from a finite sample.

From a statistical perspective, all the discussions in this paper have been at the population level. In practical statistical applications, one has to estimate  partial orthogonality relations from a sample of size $n$. To be precise, consider the cases where the underlying Hilbert space $\Hil$ has a stochastic component, that is, its inner product is defined as $\ip{x_1,y_1}_\Hil = \ex \ip{x_1,y_1}_{\Hil_0}$ for $x_1,y_1$ random elements from a \emph{base} Hilbert space $\Hc_0$, e.g, $\Hil_0 = \reals$ or $\Hil_0 = H^1([0,1])$; see the examples in Appendix~\ref{sec:pcg:examples}. Now, assume that we observe independent copies  $x^{(1)},\dots,x^{(n)}$ of  $x \in \Hil^d$, a stochastic vector in $\Hil_0^d$. Based on this sample, we wish to estimate some or all of the partial orthogonality relations \eqref{eq:partial:orth}. 

Under the perfectness assumption, according to the theorems of Section~\ref{sec:graphical:comp}, the problem reduces to estimating the support of $\Sigma^{-1}$. If $d$ is fixed and $n\to\infty$, thresholding the inverse of the \emph{sample Gram matrix}, given by
    \begin{align*}
    \Sigh = (\Sigh_{i,j}) \in \reals^{d \times d}, \quad \Sigh_{i,j}=\frac1n \sum_{k=1}^n \ip{x^{(k)}_i, x^{(k)}_j}_{\Hil_0},
    \end{align*}
may be sufficient without many assumptions on the distribution. For high-dimensional data ($d \gg n$), however, it would be interesting to see if some of the well-known penalized estimators such as the  \emph{graphical lasso}~\citep{friedman2008} can be used to estimate the PCG without the Gaussian assumptions. Much of the previous work in this area has been focused on the Gaussian case~\citep{meinshausen2006,yuan2007,banerjee2008,uhler2012}. Note that to recover the PCG, we need to have consistent support recovery, perhaps the strongest form of consistency in high-dimensional problems. 

Without the perfectness assumption, we have to consistently estimate the support of $\beta_j(S)$ for all $j$ and $S \subset [d]_j$, simultaneously.  This can be done, at least in theory, by solving all the neighborhood regression problems:
    \begin{align*}
    \widehat\coef_j(S) := \argmin_{\beta \,\in\, \R^d, \;\supp(\beta)\, \subset\, S}
    \frac1n \sum_{k=1}^n \norm{\dvec_{j}^{(k)} - \beta^T\dvec^{(k)}}_{\Hil_0}^{2},
    \end{align*}
and asking whether the support of $\widehat\coef_j(S)$, after thresholding or adding regularization, is consistent for that of $\coef_j(S)$ for all $j$ and $S$. This is a very challenging problem, as there are \emph{a priori} a super-exponential number of these neighborhood regressions. Here is where the lattice property from Section~\ref{sec:nbhd:lattice} comes into play: By exploiting the lattice property, we can reduce the total number of neighborhood regression problems that we need to look at, as detailed in Section~\ref{sec:learning:directed}.

Finally, we note that the results presented here offer an interesting connection between undirected and directed graphs via the well-understood and intuitive notion of neighborhood regression. While such connections have appeared in passing in previous work, our results confirm and extend these ideas in a very general setting. While there is much interest in using undirected graphs to simplify learning directed graphs, we emphasize that these results also offer a way to go in the \emph{opposite} direction. Namely, given a DAG, how can we interpret the relationships encoded in the model in terms of the perhaps more familiar undirected PCG? Theorems~\ref{thm:sep:char} and~\ref{thm:conn:comp} provide an explicit description of how the connections in $B$---given by $m_{j}(S_j)$---can be characterized in terms of the underlying PCG. Thus, our results provide an intuitive bridge between SEM and PCGs in a very general setting.

\section{Proofs of the main results}

Let us start with some preliminary lemmas. Recall the lattice of projections, $\Pc(\Hil)$, introduced in Section~\ref{sec:proj:lattice} and its meet and join operations. We have:
\begin{lem}\label{lem:sup:P:Q}
    For  $P,Q \in \Pc(\Hil)$,  $P y = 0$ and $Q y = 0 \iff (P \vee Q) y =0$.     
\end{lem}
In addition, for the collection of projection operators $\{P_S, S \subset [d]\}$ introduced in Definition~\ref{defn:pcg}, we have:  For any $T_1,T_2 \subset [d]$, $P_{T_1} \vee P_{T_2} = P_{T_1 \cup T_2}$ and $P_{T_1} \wedge P_{T2} = P_{T_1 \cap T_2}$. In particular,  $S \subset T$ implies $P_S \le P_T$. 
All these statements extend to any number of operators and can be argued by considering bases for the underlying subspaces. For example, for the statement in Lemma~\ref{lem:sup:P:Q}, a basis for the range of $P \vee Q$ can be built in such a way that it consists of the union of bases for the range of $P$ and  the range of $Q$, from which the statement follows.
The following lemma is also key in the proofs: 

\begin{lem}\label{lem:inclusion}
    For $S \subset T$,  $R = T \setminus S$ and any $L \in B(\Hil)$,
    \begin{align*}
    P_T L = P_S L \iff P_R P_S^\perp L = 0.
    \end{align*}
\end{lem}
\begin{proof}
    Since $P_S \le P_T$, we have $P_T P_S = P_S$, hence 
    $(P_T - P_S) L = P_{T} (I - P_S) L= P_{T} P_S^\perp L$. That is, the LHS is equivalent to $P_{T} P_S^\perp L = 0$. Multiplying by $P_R$ and using $P_R \le P_{T}$ gives the RHS, i.e., $P_R P_{S}^\perp L = 0$. Now assume that the RHS is true; since in addition we have $P_S P_S^\perp L = 0$, it follows that $(P_R \vee P_S) P_S^\perp L = 0$ by  Lemma~\ref{lem:sup:P:Q}, from which we get the LHS by noting that $P_R \vee P_S = P_{S \cup R}$. 
\end{proof}

\subsection{Proof of Theorem~\ref{thm:lattice}}\label{sec:proof:thm:lattice}
\emph{Closure under intersections:} Let $Q$ be the projection onto the range of $P_S P_j$, so that we have $Q \le P_S$. Assume that $T_1,T_2,\dots \in \Tc_j(S)$ and let $R = \bigcap_{k} T_k$. Note that $Q$ is also the projection onto the range of $P_{T_k} P_j = P_S P_j$ for all $k$. Hence, $Q \le P_{T_k}$ (and also $Q P_{T_k} P_j = P_{T_k} P_j$). Since $P_R = \bigwedge_{k} P_{T_k}$, it follows that $Q \le P_R \le P_{T_k}$ for all $k$. This last statement is equivalent to $ P_R Q = Q$ and $P_R P_{T_k} = P_R$ for all $k$.
Now, we have 
\begin{align}
    P_R P_j = P_R P_{T_k} P_j = P_R (Q P_{T_k} P_j) = Q P_{T_k} P_j = P_{T_k} P_j = P_S P_j,
\end{align}
hence $R \in \Tc_j(S)$. The first equality follows from $P_R \le P_{T_k}$, the second and fourth since $Q$ projects onto range $P_{T_k} P_j$, and the third since $Q \le P_R$. Thus, we have shown that $\Tc_j(S)$ is closed under intersections.

\smallskip
\emph{Closure under unions:} 
Let $\Ss$ be the minimal element of $\Tc_j(S)$ which exists by the previous argument. Note that
\begin{align}\label{eq:temp:451}
     T \in \Tc_j(S) \iff T \supset \Ss \;\text{and}\;P_T P_{\Ss}^\perp P_j = 0.
\end{align}
To see this, it is enough to note that for $T \supset \Ss$, we have $(P_T - P_{\Ss})P_j = P_T (I - P_{\Ss})P_j = P_T P_{\Ss}^\perp P_j$, where the first equality is by $P_{\Ss} \le P_T$. Now, assume that $T_1,T_2,\dots \in \Tc_j(S)$ and let $M = \bigcup_{k} T_k$. Since $T_k \in \Tc_j(S)$, we have  $T_k \supset \Ss$ and $P_{T_k} P_{\Ss}^{\perp} P_j = 0$ for all $k$. It follows that $(\bigvee_k P_{T_k } )P_{\Ss}^{\perp} P_j = 0$ (Lemma~\ref{lem:sup:P:Q}). But $P_M = \bigvee_k P_{T_k }$, that is, $ P_M P_{\Ss}^{\perp} P_j = 0$, hence $M \in \Tc_j(S)$ by~\eqref{eq:temp:451}. 

\smallskip
\emph{Convexity:} Let $S_1,S_2 \in \Tc_j(S)$ and $S' \subset [d]_j$. Assume that $S_1 \subset S' \subset S_2$, so that $P_{S_1} \le P_{S'} \le P_{S_2}$. We have $P_{S_1} P_j = P_{S_2} P_j$. Multiply on the left by $P_{S'}$ and note that $P_{S'} P_{S_1} = P_{S_1}$, and $P_{S'} P_{S_2} = P_{S'}$. Thus, $P_{S_1} P_j = P_{S'} P_j$ showing that $S' \in \Tc_j(S)$. The proof is complete.

\medskip
\begin{rem}
	The argument above in terms of the projections works for an infinite-dimensional Hilbert space $\Hil$ and infinitely many variables $\{x_1,x_2,\dots\}$. It also holds even when the variables are dependent (i.e., $\Sigma$ is rank-deficient) if we remain at the level of projections (i.e., not map projections onto sets of variables); see for example the general result in Appendix~\ref{sec:abs:lattice}.
	
	In the finite-dimensional case $\Hil \cong \reals^n$, we can use the following simple argument instead to show that $\Tc_j(S)$ is a lattice:
	Suppose that $T_1,T_2 \in \Tc_j(S)$, i.e., $P_{T_1}x_j=P_{T_2}x_j =: \hat{x}_j$. Then 
	\begin{align*}
		\hat{x}_j\in \text{span}(x_{T_1\cap T_2}), \quad \text{and} \quad  x_j-\hat{x}_j \perp \text{span}(x_{T_1\cup T_2})
	\end{align*}
	where the first statement uses $\Sigma\succ 0$.
	 These two imply that $\hat x_j$ is the projection of $x_j$ onto $\text{span}(x_{T_1\cap T_2})$ as well as $\text{span}(x_{T_1\cup T_2})$, that is,  $T_1 \cap T_2 \in \Tc_j(S)$ and so is $T_1\cup T_2$.
\end{rem}

\subsection{Proofs of Section~\ref{sec:enumerate:POs}}\label{sec:proof:enumerate:POs}

We start with a result that provides alternative characterization of partial orthogonality (Definition~\ref{defn:PO}) in terms of the SEM coefficients (Definition~\ref{defn:nhbd}). 
It is not hard to see that $[\beta_j(S)]_S = \Sigma_S^{-1} \Sigma_{S,j}$ (Appendix~\ref{sec:proof:proj:Sigma:gen}). In addition, it is possible to write the condition $P_A P_S^\perp P_B =0$ directly in terms of the SEM coefficients, and also in terms of the Gram matrix $\Sigma$. We have the following result proved in Appendix~\ref{sec:proof:proj:Sigma:gen}:

\begin{prop}\label{prop:proj:Sigma:gen}
	Under~\eqref{asum:PSD:Sigma}, the following are equivalent:
	\begin{enumerate}[label=(\alph*)]
		\item $P_A P_S^\perp P_B = 0$, \label{eq:PAPSperpPB}
		\item $ 
		[\beta_j(A)]_i = \sum_{k \in S}  [\beta_j(S)]_k [\beta_k(A)]_i, \quad \forall i \in A,\; j \in B$, \label{eq:system:of:beta}
		\item $\Sigma_{B,A} - \Sigma_{B,S} \Sigma_{S}^{-1} \Sigma_{S,A} = 0$ (replaced with $\Sigma_{B,A} = 0$ if $S = \emptyset$),
		\item $[\beta_{j}(Si)]_i  = 0, \; \forall i \in A, \;j\in B$. 
		\label{eq:pairwise:beta}
	\end{enumerate}
	Moreover, for any $i,j \in [d]$ and $S \subset [d]_{ij}$, we have
	\begin{align}\label{eq:explit:pairwise:SEM}
	[\beta_j(Si)]_i = \frac{\ip{x_i, P_S^\perp x_j}_\Hil}{ \ip{x_i,P_S^\perp x_i}_\Hil}.
	\end{align}
\end{prop}
We note that $(d)$ can be equivalently written as $[\beta_{j}(Si)]_i =[\beta_i(Sj)]_j = 0, \; \forall i \in A, j\in B$, due to symmetry. Most assertions in Proposition~\ref{prop:proj:Sigma:gen} are colloquially known (perhaps except part~(b)). For example, expressions similar to~\eqref{eq:explit:pairwise:SEM} have appeared before~\citep{Tian2007}, going back to the work of~\cite{cramer1946} and~\cite{dempster1969}.

\paragraph{Proof of Proposition~\ref{prop:all:POs:for:Tj}}
	By assumption $Ti := T \cup \{i\} \in \Tc_j = [m_j, M_j]$ and $T \cap \{i,j\} = \emptyset$.
	By Proposition~\ref{prop:proj:Sigma:gen}(d), $(j \perp i \mid T)$ holds if and only if
	\begin{align*}
	[\beta_j(T i)]_i = 0 \stackrel{(a)}{\iff} i \notin m_j \stackrel{(b)}{\iff} i \in M_j \setminus m_j
	\end{align*}
	where (a) is due to $\supp(\beta_j(Ti)) = m_j$ by the lattice property, and (b) is due to assumption $i \in M_j$. The proof is complete.
\endproof

\paragraph{Proof of Theorem~\ref{thm:all:POs}}
	Consider  statement $j \perp i \mid T$ which by convention implies $T \cap \{i,j\} = \emptyset$. Since $\big\{[m_j^\ell,M_j^\ell], \ell=1,\dots,K_j \big\}$ is a partition of $2^{[d]_j}$, there is a unique $k \in [K_j]$ such that $T \cup \{i\}$ belongs to $[m_j^k, M_j^k]$. This implies that $T \in [m_j^k, M_j^k \setminus \{i\}]$ and $ i \in M_j^k$. By Proposition~\ref{prop:all:POs:for:Tj}, $j \perp i \mid T$  holds iff $i \in M_j^k \setminus m_j^k$. The proof is complete.
\endproof

\subsection{Proofs of Section~\ref{sec:comp}}\label{sec:proof:comp}

The proof of Proposition~\ref{prop:lat:comp:complex} is immediate by an inspection of Algorithm~\ref{alg:lattice:comp} and Theorem~\ref{thm:poly:N} follows by combining Proposition~\ref{prop:lat:comp:complex} and Lemma~\ref{lem:decide:set:system}.

\paragraph{Proof of Lemma~\ref{lem:decide:set:system}}
	There are multiple algorithms to solve this problem. Here we describe one due to Brendan McKay. 
	To simplify notation let $r := d-1 = |[d]_j|$ and identify $[d]_j$ with $[r]$ by relabeling if need be. We also drop the subscript $j$ for from $\Tc_j^\ell$ and so on.
	Given disjoint lattices $\Tc^\ell = [m^\ell,M^\ell] \in 2^{[r]}, \; \ell=1,\dots,K$, and any set $S \subset [r]$, let 
	\begin{align*}
	Q(S) = \sum_{\ell=1}^K Q_\ell(S), \quad Q_\ell(S) = | \{ T \in \Tc^\ell :\; S \subset T\} |.
	\end{align*}
	Computing $Q_\ell(S)$ is easy and we consider it an atomic set operation. In particular, if $S \not \subset M^\ell$ then $Q_\ell(S) = 0$; otherwise, $Q_\ell(S) = 2^{|(M^\ell \setminus m^\ell) \setminus S|}$. Then computing $Q(S)$ can be done by at most $O(K)$ set operations.
	Note that for any set $S \subset [r]$ we have $Q(S) \le 2^{r-|S|}$ with equality if and only if all possible subsets of $[r]$ containing $S$ are present in $\Tcb := \cup_{\ell=1}^K \Tc^\ell$.
	
	Now, if $Q(\emptyset) = 2^{r}$, then power set $2^{[r]}$ is covered. Otherwise, there is $i_1 \in [r]$ such that $Q(\{i_1\}) < 2^{r-1}$ which can be found by trying all the $r$ elements of $[r]$. Then, one can find $\{i_1,i_2\}$ such that $Q(\{i_1,i_2\}) < 2^{r-2}$ by trying all $r$ possible cases for $i_2$. Continuing in this manner, we have either of the following two: 
	\begin{enumerate}
		\item After at most $r-1$ rounds, we have $S \subset [d]_j$ such that $Q(S) = 0$. This $S$ has the desired property and the algorithm terminates.
		\item We have $S = \{i_1,\dots,i_{r-1}\}$, with $Q(S) < 2$. Since we are not in Case~1, we should have $Q(S) = 1$. This means that either (a) $S \in \Tcb$ in which case the ground set $[r]$ cannot belong to $\Tcb$, 		hence we output $[r]$, or  (b) the ground set $[r] \in \Tcb$ in which case the original $S$ cannot belong to $\Tcb$ and we output $S$.
	\end{enumerate}
	In either case, the overall complexity is at most $O(r^2 K)$ set operations.
\endproof

\subsection{Proof of Theorem~\ref{thm:sep:char}}
\label{sec:seperation}
Throughout the proof of this theorem, we work under the assumption of $\Hil$-Markov perfectness, i.e., all the lemmas in this subsection are stated under that assumption.
We note that
\begin{align}\label{eq:exchange}
    S' \in \Tc_j(S) \iff \Tc_j(S) = \Tc_j(S') \iff S \in \Tc_j(S').
\end{align}
\begin{lem}\label{lem:sep:1}
    Let $A \subset S \subset  [d]_j$. Then, $A \in \Tc_j(S)$ iff $A$ separates $j$ from $S\setminus A$.
\end{lem}
\begin{proof}
    Let $R = S \setminus A$. Then, by $\Hil$-Markov perfectness, $A$ separates $j$ from $S\setminus A$ iff $P_R \perp P_j \mid P_A$. This means $P_R P_A^\perp P_j = 0$ which is equivalent to $P_A P_j = P_S P_j$, by Lemma~\ref{lem:inclusion}.
\end{proof}

We can now write $m_j(S)$, the minimum element of $\Tc_j(S)$, as
\begin{align*}
    m_j(S) =
     \bigcap \{A \subset S: \; A \in \Tc_j(S) \} = \bigcap \big\{A \subset S:\; \text{$A$ separates $j$ and $S \setminus A$.} \big\}
\end{align*}
the first equality is because we can restrict to subsets of $S$ when finding the minimum, due to $\Tc_j(S)$ being closed under intersections (lattice property), and the second is by Lemma~\ref{lem:sep:1}. This proves part~(a) of Theorem~\ref{thm:sep:char}.

\begin{lem}\label{lem:sep:2}
    Let $B \subset [d]_j$ and $A \in \Tc_j(S)$ be disjoint. Then $A$ separates $B$ and $j$ if and only if $A \cup B \in \Tc_j(S)$.
\end{lem}
\begin{proof}
     Since $A \in  \Tc_j(S)$, we have $\Tc_j(A) = \Tc_j(S)$.  Hence, $A \cup B \in \Tc_j(S)$ is equivalent to $A \cup B \in \Tc_j(A)$ which is equivalent to $A \in \Tc_j(A \cup B)$ by~\eqref{eq:exchange}. By Lemma~\ref{lem:sep:1}, this latter statement is equivalent to $A$ separates $j$ and $B$.
\end{proof}
Since $m_j(S)$ is the smallest element of $\Tc_j(S)$, we get the following characterization:
\begin{align}\label{eq:sep:char}
    \Tc_j(S) = \{ T \subset [d]_j:\; \text{$m_j(S)$ separates $j$ and $T \setminus m_j(S)$}\}.
\end{align}
To see this, fix $S$ and let $S^* = m_j(S)$. Then $T \in \Tc_j(S)$ iff $S^{*} \subset T$ by minimality 
and $S^*$ separates $j$ and $T \setminus S^*$ by Lemma~\ref{lem:sep:2}.
From~\eqref{eq:sep:char}, it follows that $M_j(S) = m_j(S) \cup B$ where $B$ is the largest set that can be separated from $j$ by $m_j(S)$. This proves part~(b). Part~(c) also follows from~\eqref{eq:sep:char}.

\subsection{Proof of Theorem~\ref{thm:conn:comp}}

Let us introduce some new notations. Recall that we write $\sep(A,C,B)$ for disjoint sets $A,B$ and $C$ when $C$ separates $A$ and $B$. This means that any path from a node $i \in A$ to node $j \in B$ should pass through a node $k \in C$. This includes the case where there is no path from $A$ to $B$.

\begin{figure}[t]
    \centering
    \subfloat[]{
    \begin{tikzpicture}
                \begin{scope}[every node/.style={circle,thick,draw, inner sep =.5mm}]
                    \node (A) at (0,0) {$A$};
                    \node (C_1) at (1,.6) {$C_1$};
                    \node (C_2) at (1,-.6) {$C_2$};
                    \node (B) at (2,0) {$B$};   
                \end{scope}

                \draw[dashed] (A) -- (C_1);
                \draw[dashed] (A) -- (C_2);
                \draw[thick] (C_1) -- (C_2);
            
            \end{tikzpicture}   
    }       
    \qquad
    \subfloat[]{
        \begin{tikzpicture}
            \begin{scope}[every node/.style={circle,thick,draw, inner sep =.5mm}]
                \node (A) at (0,0) {$A$};
                \node (C_1) at (1,.6) {$C_1$};
                \node (C_2) at (1,-.6) {$C_2$};
                \node (B) at (2,0) {$B$};
            \end{scope}

            \draw[dashed] (A) -- (C_1);
            \draw[dashed] (A) -- (C_2);
            \draw[dashed] (C_1) -- (B);
        \end{tikzpicture}   
        \label{fig:1:a}
    }

\caption{Illustration of the reduction property. 
    We break the cases respecting $\sep(A,C_1 \uplus C_2,B)$ into two groups, depending on whether there is a path between $C_1$ and $C_2$ (a), or  not (b). The dashed lines are possible paths.  In none of the cases there could be a path between $C_2$ and $B$ because of the assumption $\sep(C_2,A,B)$. The same assumption precludes a path between $C_1$ and $B$ in cases in (a), because of the existence of a path from $C_2$ to $C_1$. As seen from this figure, in all possible cases, $A$ and $B$ are separated by $C_1$.
 \label{fig:reduct:property}}\end{figure}
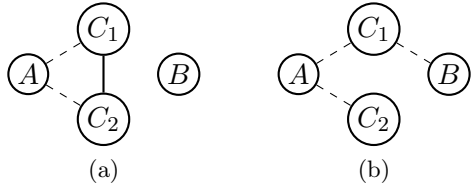

 \begin{lem}
    Separation has these properties:
    \begin{itemize}
        \item \emph{Additive property}:  $\sep(A,C_1,B_1)$ and $\sep(A,C_2,B_2)$ implies $\sep(A,(C_1\cup C_2),(B_1 \cup B_2))$.

        \item \emph{Reduction property}: $\sep(A,C_1 \uplus C_2,B)$ and $\sep(C_2,A,B)$ implies $\sep(A,C_1,B)$.
    \end{itemize}
        
 \end{lem}

\begin{proof}
    We only show the reduction property, for which it is enough to show that there is no path from $A$ to $B$ that passes through $C_2$ and not $C_1$. If so, a portion of it gives a path from $C_2$ to $B$ that does not pass through $C_1$, and clearly not $A$ since the paths do not self-cross. But then $A$ does not separate $C_2$ and $B$, a contradiction. Figure~\ref{fig:reduct:property} illustrates the argument.
\end{proof} 

\paragraph{Proof of Theorem~\ref{thm:conn:comp}}
    By induction we can reduce to the case where the nodes are partitioned into two disjoint 
    components  $G_1$ and $G_2$, after the removal of $j$. This implies $\sep(G_1,j,G_2)$. For any set $A$, let $A_k := A \cap G_k, k =1,2$. We have
    \begin{align}\label{eq:sep:decomp}
        \sep(j,A,(S\setminus A)) \quad \iff\quad  \sep(j,A_1,(S_1\setminus A_1)) \quad \text{and}\quad 
        \sep(j,A_2,(S_2\setminus A_2))
    \end{align}
    To see this, note that the RHS implies $\sep(j,A, (S_1\setminus A_1)) \cup (S_2\setminus A_2)$ by the additive property. But $(S_1\setminus A_1) \cup (S_2\setminus A_2) = S \setminus A$ since $G_1$ and $G_2$ are disjoint. 

    Now assume the LHS. Then,  clearly $\sep(j,A,(S_1\setminus A_1))$. But we also have $\sep(A_2,j,(S_1\setminus A_1))$. Applying the reduction property, we get $\sep(j,A_1,(S_1\setminus A_1))$. The other implication is similar, hence we get the RHS.

    Equipped with~\eqref{eq:sep:decomp}, we can prove part~(a). Let $\Ss$ be the smallest subset of $S$ separating $j$ and $S \setminus \Ss$ (see \eqref{eq:Sstar}), and let $\Ss_1$ and $\Ss_2$ be its components (i.e, $\Ss_k = \Ss \cap G_k,\;\; k = 1,2$). Then, by~\eqref{eq:sep:decomp} $\Ss_k$ separates $j$ and $S_k \setminus\Ss_k$, for $k=1,2$. It remains to show that $\Ss_k$ is the smallest such set (for each $k$). Suppose that there is a proper subset $S'_1$ of $\Ss_k$ that separates $j$ and $S \setminus S'_1$. Then, applying~\eqref{eq:sep:decomp} with $A_1 = S'_1$ and $A_2 = \Ss_2$ and $A = S':= S'_1 \cup \Ss_2$, we conclude that $S'$ separates $j$ and $S\setminus S'$. But $S'$ is a proper subset of $\Ss$, violating the assumption that $\Ss$ is the minimal set with such property. It follows that we should have $m_j(S_1) = \Ss_1$ and $m_j(S_2) = \Ss_2$ which proves part~(a).

    For part~(b), let $\Ss_k = m_j(S_k), k =1,2$ and $\Ss = m_j(S)$. From part~(a), we have $\Ss = \Ss_1 \cup \Ss_2$. Fix $r \in [d]_j$. We claim that  
    \begin{align}\label{eq:sep:decomp:2}
        \sep(r,\Ss,j) \quad \iff \quad \sep(r,\Ss_1,j) \;\;\text{or}\;\; \sep(r,\Ss_2,j).
    \end{align}
    The $\Leftarrow$ implication is clear. For the other direction, assume that $r$ is separated from $j$ by $\Ss$, i.e., all the paths from $r$ to $j$ pass through $\Ss$. WLOG, assume that $r \in G_1$ (the case $r \in G_2$ is similar). Suppose that there is a path from $r$ to $j$ that passes through $\Ss_2$. A portion of this path gives a path from $r$ to $\Ss_2$, hence to $G_2$, that does not pass through $j$, contradicting $\sep(G_1,j,G_2)$. Hence all the paths from $r$ to $j$ should pass through $\Ss_1$, i.e., $\sep(r,\Ss_1,j)$, proving~\eqref{eq:sep:decomp:2}.
    From~\eqref{eq:sep:decomp:2}, we conclude that $E_j(\Ss) = E_j(\Ss_1) \cup E_j(\Ss_2)$. Combined with characterization of $M_j(S)$ in Theorem~\ref{thm:sep:char}(b), this proves the first equality in part(b).
    
    To see the second equality in~(b), recall that $\Tc_j(S_k;G_k)$ is the lattice with ground set restricted to $G_k \cup \{j\}$. First, the minimal element of $\Tc_j(S_k;G_k)$ is the same as that of $\Tc_j(S_k)$. This is true since a subset $A$ of $S_k$ separates $j$ and $S_k \setminus A$ in $G_k \cup \{j\}$ iff it does so in $G$. (If the separation happens in $G_k \cup \{j\}$ but not in $G$, there will be a path from $S_k \setminus A$ to $j$ passing through some $G_r$, $r\neq k$ contradicting disjointness of $G_k$ and $G_r$.) Even more directly, a minimal subset of $\Tc_j(S_k)$ is a subset of $S_k$, hence $G_k$, and the restriction in  $\Tc_j(S_k;G_k)$ is automatically satisfied.  Knowing that the minimal element of $\Tc_j(S_k;G_k)$ is $S^*_k$, we invoke Theorem~\ref{thm:sep:char}(b) restricted to $G_k \cup \{j\}$ to conclude that $M_j(S_k;G_k) = S_k^* \cup E_j(S_k^*;G_k)$ where $E_j(S_k^*;G_k) = \{i \in G_k:\;  \text{$S^*_k$ separates $i$ and $j$}  \}$. Going through the argument leading to~\eqref{eq:sep:decomp:2}, it is clear that we can replace~\eqref{eq:sep:decomp:2} with
    \begin{align}
         \sep(r,\Ss,j) \quad \iff \quad [\sep(r,\Ss_1,j) \;\text{and}\;r \in G_1 ] \;\;\text{or}\;\; [\sep(r,\Ss_2,j)  \;\text{and}\;r \in G_2 ]
    \end{align}
    showing that $E_j(\Ss) = E_j(\Ss_1;G_1) \, \cup\, E_j(\Ss_2;G_2)$. Combined with the expression for $M_j(S_k;G_k)$, the desired result follows.
\endproof

\section*{Acknowledgment}
This work was supported by NSF grant IIS-1546098.

\printbibliography

\appendix

\section{Details}\label{app:detials}

\subsection{Details of the neighborhood lattice}\label{sec:details:nbhd:lat}
\begin{lem}\label{lem:equiv:lattice:defs}
	Representations~\eqref{eq:lattice:proj:def} and~\eqref{eq:lattice:alt} are equivalent.
\end{lem}

\begin{proof}[Proof of Lemma~\ref{lem:equiv:lattice:defs}]
	Let $\Tc_j(S)$ be as defined in~\eqref{eq:lattice:proj:def}.
	To see the first equality in~\eqref{eq:lattice:alt}, note that $\beta_j(S)^\tpose \dvec = P_S \dvec_j$. Since the Gram matrix of $\dvec$ is positive definite, $\beta_j(S) = \beta_j(T)$ is equivalent to $\beta_j(S)^\tpose \dvec = \beta_j(T)^\tpose \dvec$, and hence to $P_S \dvec_j = P_T \dvec_j$. The result follows since for any operator $L$, $L \dvec_j = 0$ is equivalent to $L P_j = 0$. (We have $L P_j y = \ip{x_j,y} L x_j$ for all $y$, assuming $\norm{x_j}_{\Hil}=1$.) The second equality in~\eqref{eq:lattice:alt} follows from the first equality and positive definiteness of the Gram matrix. 
\end{proof}

\subsection{PCG details}

\subsubsection{Partial orthogonality via Gram matrix}\label{sec:parorth:gram}
The following result is proved in Appendix~\ref{sec:proof:proj:Sigma}:
\begin{lem}\label{lem:Gram:char:pair}
	$|\Sigma_{Si,Sj}| = 0$ is equivalent to $P_i P_S^\perp P_j = 0$, for all $i,j$ and $S \subset [d]_{ij}$. \end{lem}
\noindent
It is worth noting that the lemma is not restricted to random variables and holds in the general Hilbert space setup of Definition~\ref{def:pairwise:L2:abs:PCG}. Since PCGs are preserved under Hilbert space isometries (Appendix~\ref{sec:proof:pairwise:global}), by passing to the case where $\gv$ is a zero-mean Gaussian random vector  with covariance $\Sigma$ (i.e., a special case of Example~\ref{exa:usual:rvs}), and using the known results on Gaussian pairwise conditional independence \cite[Prop.~5.2]{Lauritzen1996}, we conclude that~\eqref{eq:pairwise:L2} in Definition~\ref{defn:pcg} is in fact equivalent to
\begin{align}\label{eq:inv:Gram}
\text{$i \nsim j$ in $G$} \iff [\Sigma^{-1}]_{i,j} = 0
\end{align}
or equivalently $| \Sigma_{Si,Sj}| = 0$ where $S = [d]_{ij}$. That is, Lemma~\ref{lem:Gram:char:pair}, applied with with $S = [d]_{ij}$, recovers this well-known result. Note, however, that Lemma~\ref{lem:Gram:char:pair} is stronger than~\eqref{eq:inv:Gram}, since it establishes the equivalence for any $S \subset [d]_{ij}$.

\medskip
The ternary relation~\eqref{eq:partial:orth} among projections is also completely characterized by the Gram matrix $\Sigma$:
\begin{lem}\label{lem:Gram:char:gen}
	$\Sigma_{B,A} - \Sigma_{B,S} \Sigma_{S}^{-1} \Sigma_{S,A} = 0$ is equivalent to $P_A P_S^\perp P_B = 0$, for disjoint $A,S,B \subset [d]$.
\end{lem}
This lemma is proved as part of Proposition~\ref{prop:proj:Sigma:gen} (Section~\ref{sec:proof:enumerate:POs}). We refer to~\eqref{eq:partial:orth} as \emph{partial orthogonality}, or \emph{conditional orthogonality}, a generalization of the notion of partial uncorrelatedness of random variables.
Clearly the global $\Hil$-Markov property implies the pairwise version as a special case.
The following result, proved in Appendix~\ref{sec:proof:pairwise:global},  establishes the equivalence of the pairwise and global  $\Hil$-Markov properties when the covariance matrix is non-degenerate. Note that this is the $L^2$ analogue of a similar well-known result for CIGs \cite[Thm.~3.9, p.~35]{Lauritzen1996}.

\begin{lem}\label{lem:pairwise:global}
	Assuming $\Sigma \succ 0 $, the pairwise $\Hil$-Markov property implies the global $\Hil$-Markov property. 
\end{lem}

The PCG, as well as the pairwise and global $\Hil$-Markov properties, can be thought of as being defined based on (1) a vector $x \in \Hil^d$ with a particular Gram matrix $\Sigma$, or more abstractly based on (2) a family of projection operators on the subspaces generated by $x$. We have implicitly taken the first viewpoint, in which case we can characterize the PCG in terms of $\Sigma$. There are however subtle differences between (1) and (2). A vector $x$ uniquely identifies both a Gram matrix $\Sigma$ and a  collection of subspaces, or equivalently a collection of projection operators $\{P_A\}$. However, neither $\Sigma$ nor $\{P_A\}$ uniquely identifies the other. Nevertheless both are useful in encoding $\Hil$-Markov properties as illustrated in Lemma~\ref{lem:Gram:char:pair}.
For example, it is clear from the projection viewpoint, that rescaling any of the components of $x$ does not change the subspaces and hence the $\Hil$-Markov properties. This implies the following:
\begin{lem}\label{lem:diag:rescaling}
	$D \Sigma D$ has the same $\Hil$-Markov properties as $\Sigma$ for any diagonal matrix $D$, with nonzero diagonal entries.
\end{lem}
Although this can be verified directly in terms of $\Sigma$, using the characterization of Lemma~\ref{lem:Gram:char:gen}, the projection viewpoint makes this immediately clear. A major theme of this paper is showing the usefulness of the projection interpretation in understanding the algebraic features of the $\Hil$-Markov property.

\subsubsection{PCG Examples}\label{sec:pcg:examples}
Let us now give more concrete examples of the abstract PCG of Definition~\ref{def:pairwise:L2:abs:PCG}, which go beyond the familiar case of Example~\ref{exa:usual:rvs}, and also show the utility of Lemma~\ref{lem:Gram:char:pair}.

\begin{figure}[t]
	\def\pcgscale{.16}
	\begin{minipage}{.8\textwidth}
		\includegraphics[width=\pcgscale\linewidth]{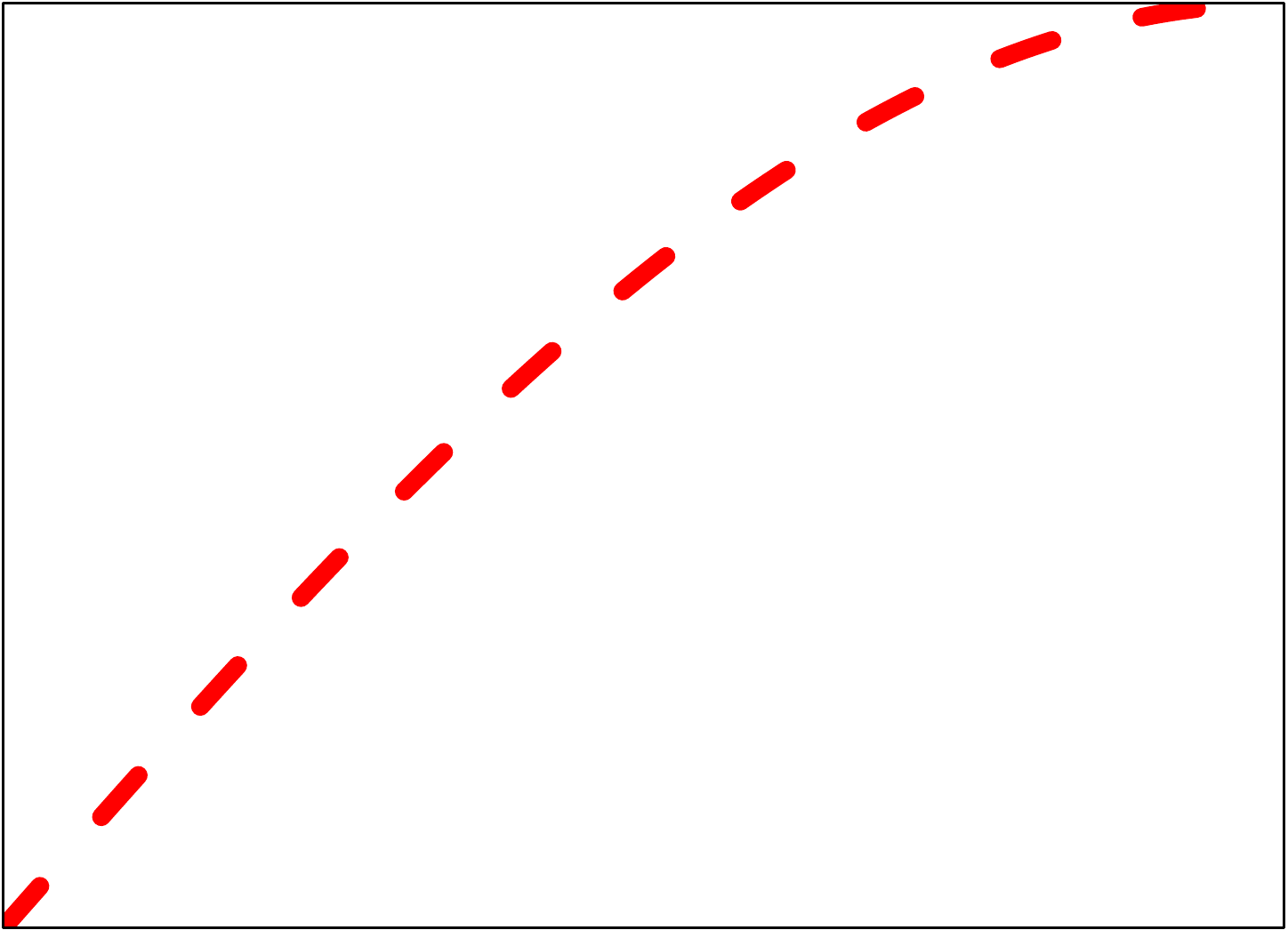}
		\includegraphics[width=\pcgscale\linewidth]{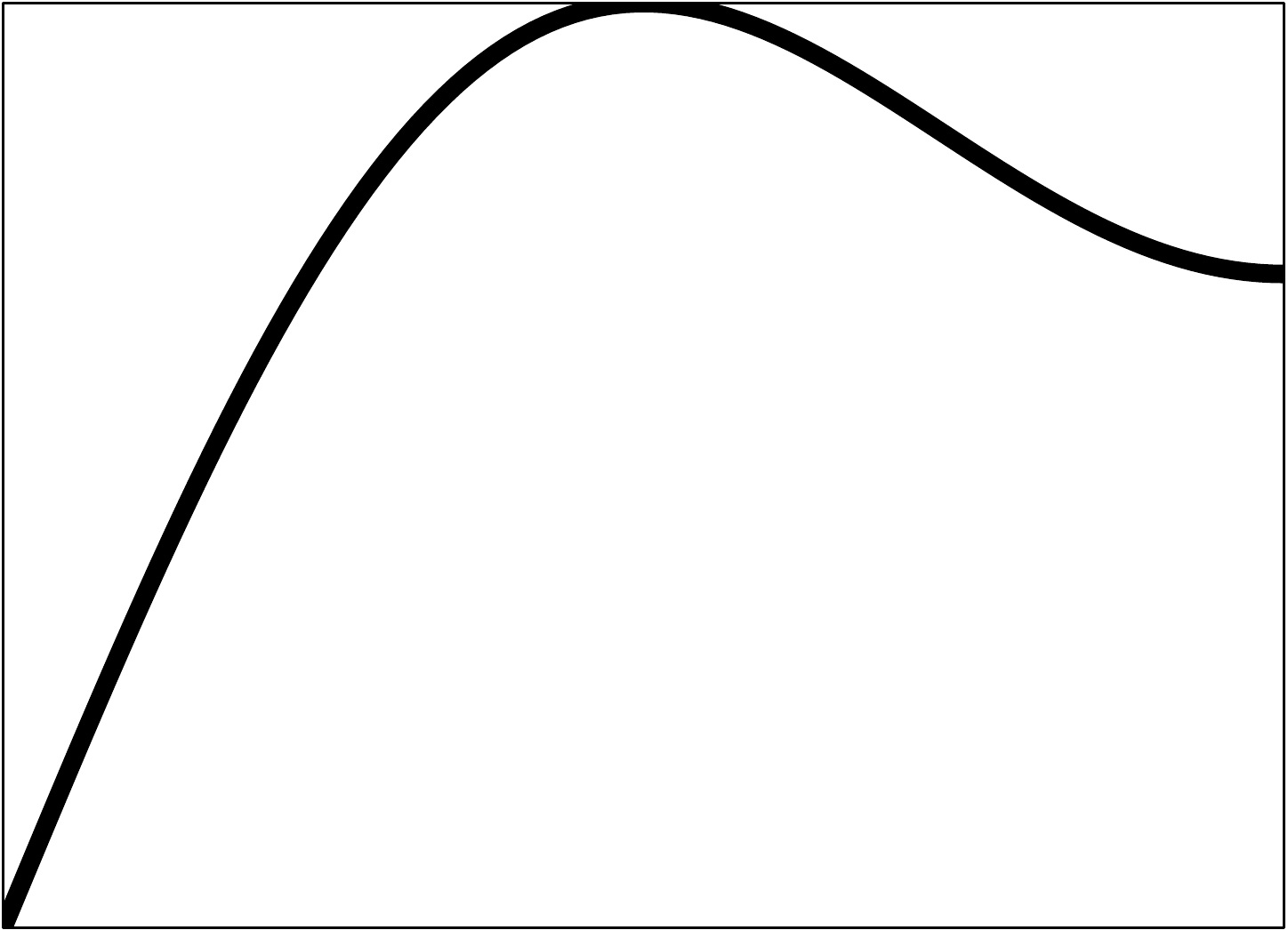}
		\includegraphics[width=\pcgscale\linewidth]{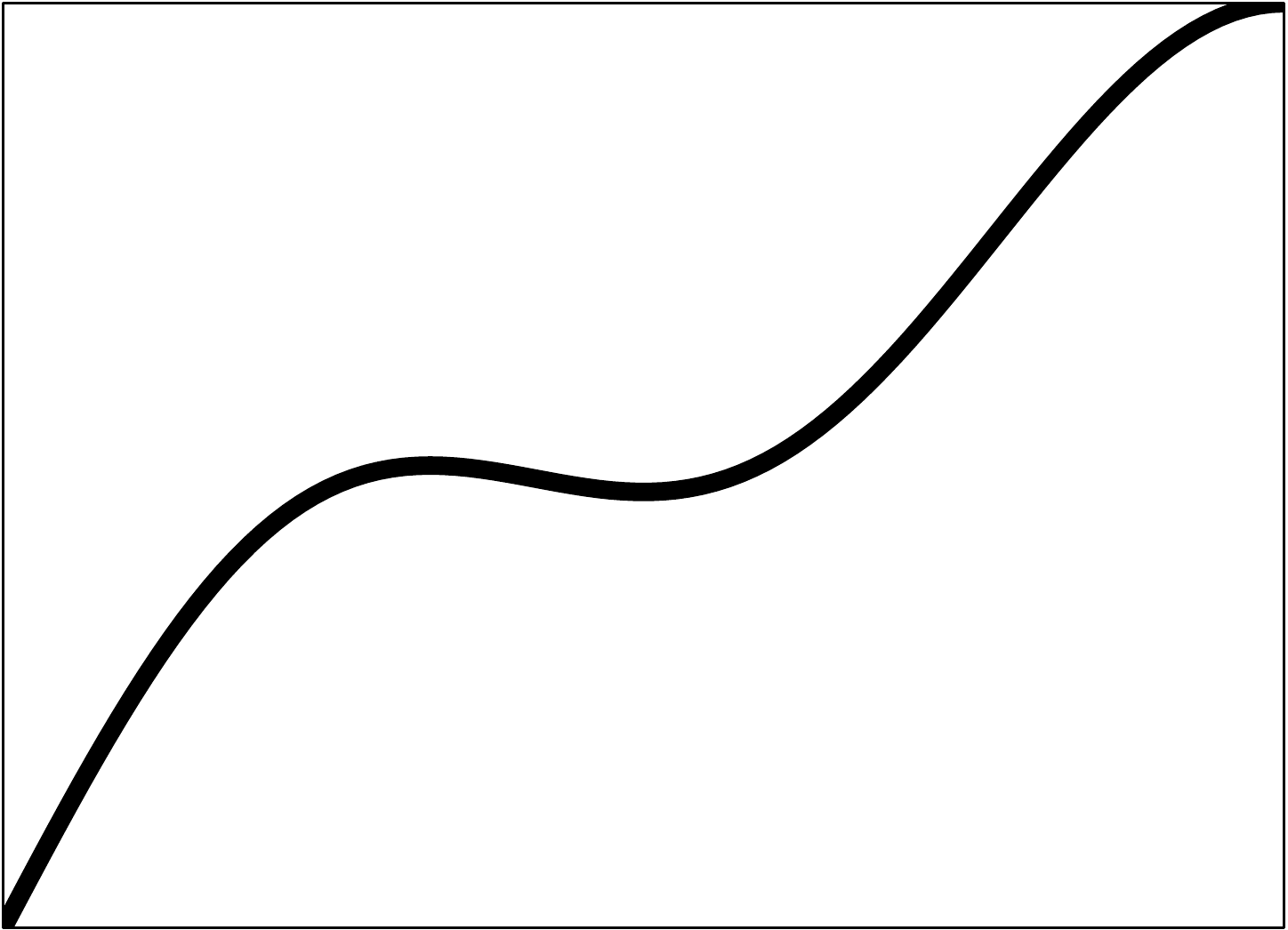}
		\includegraphics[width=\pcgscale\linewidth]{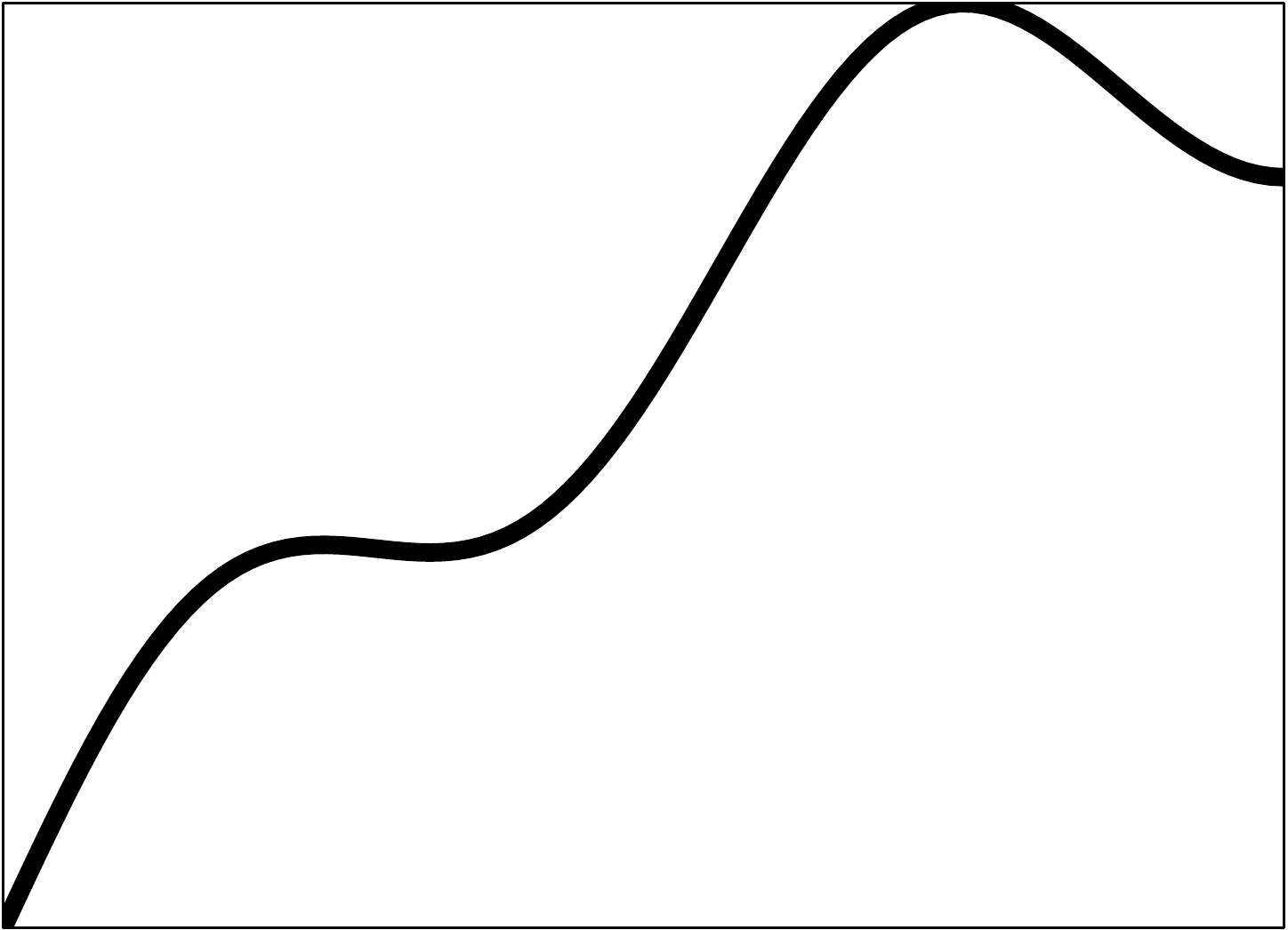}
		\includegraphics[width=\pcgscale\linewidth]{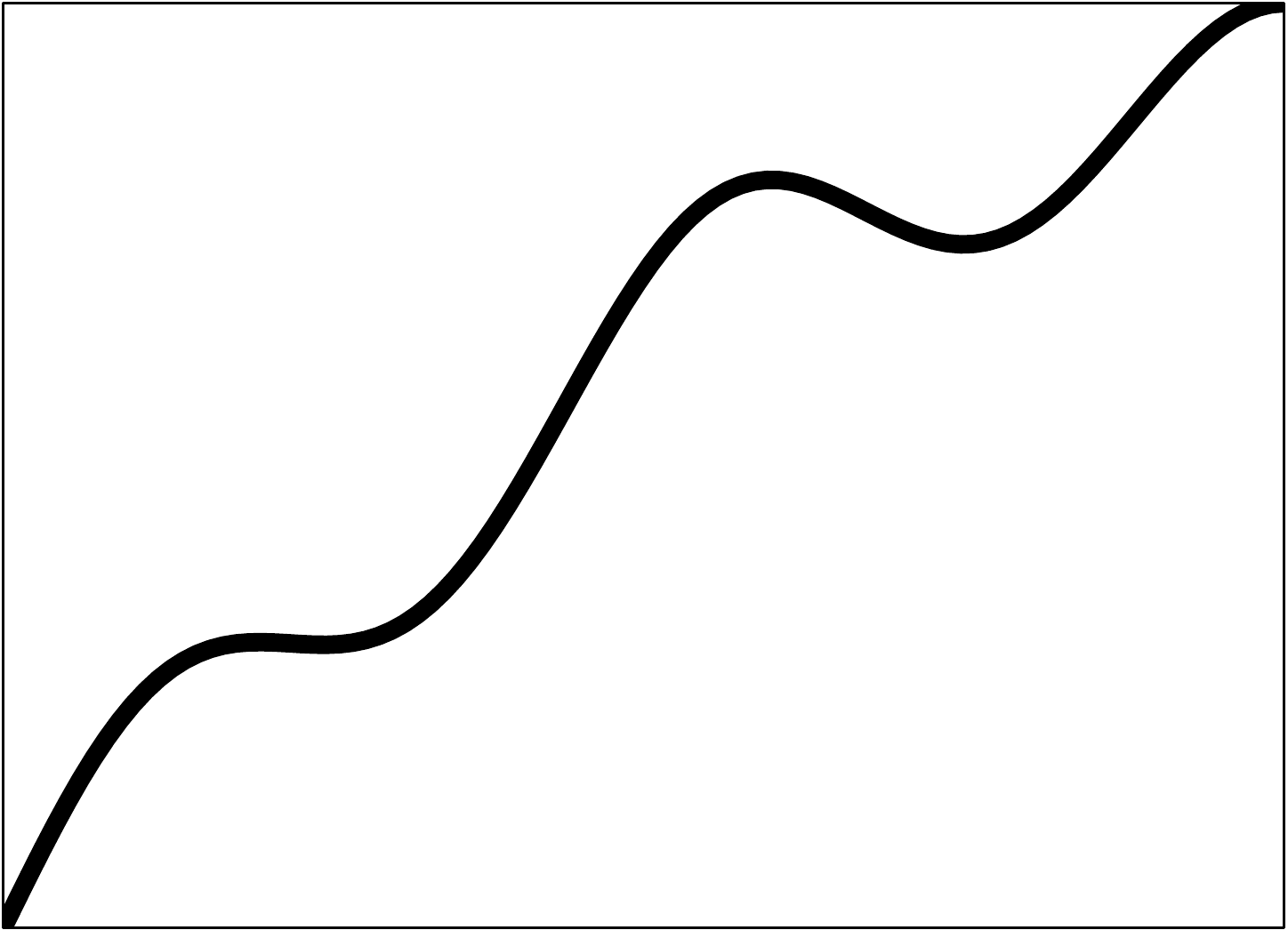}
		\includegraphics[width=\pcgscale\linewidth]{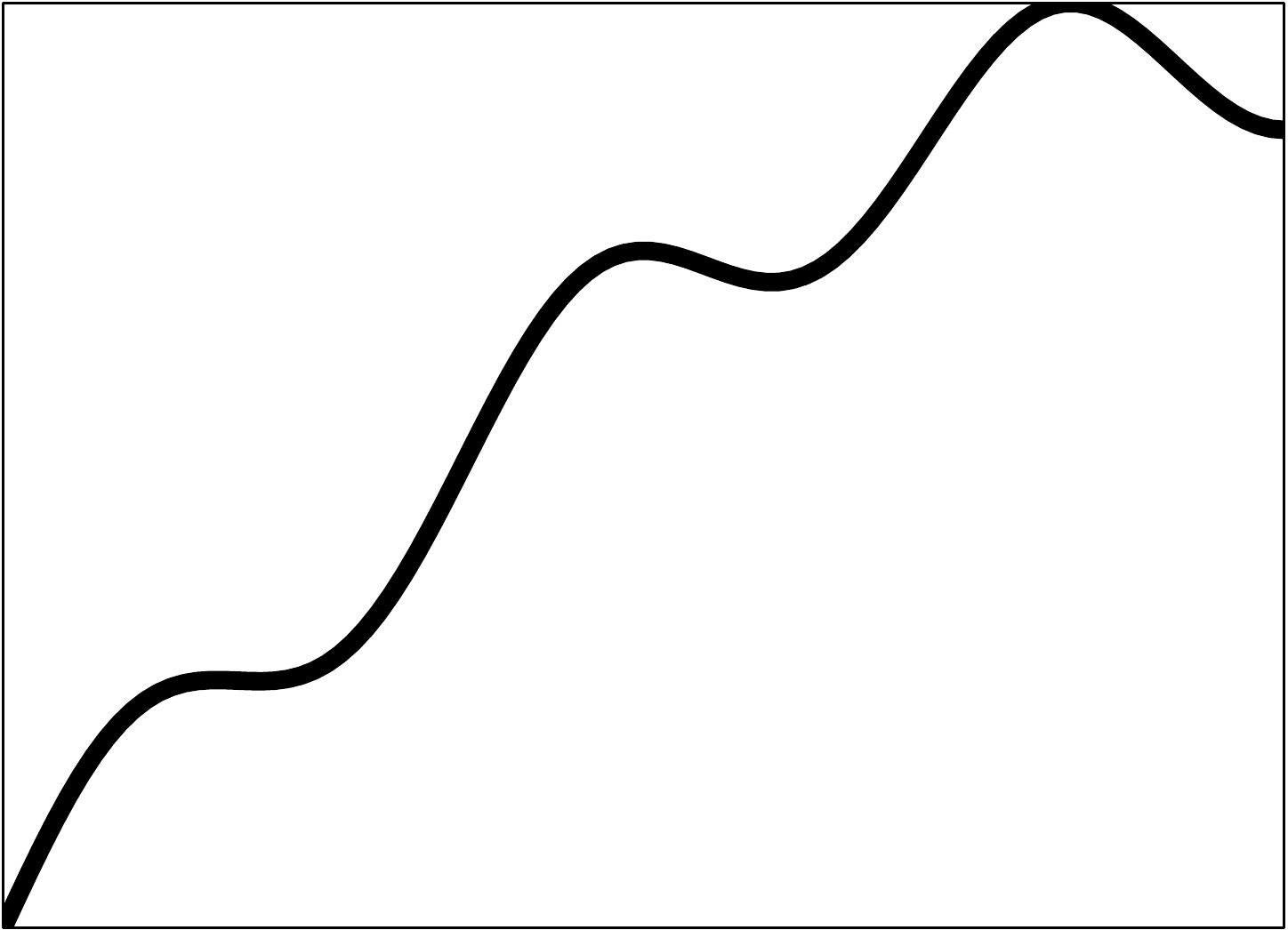}\\[1ex]
		\includegraphics[width=\pcgscale\linewidth]{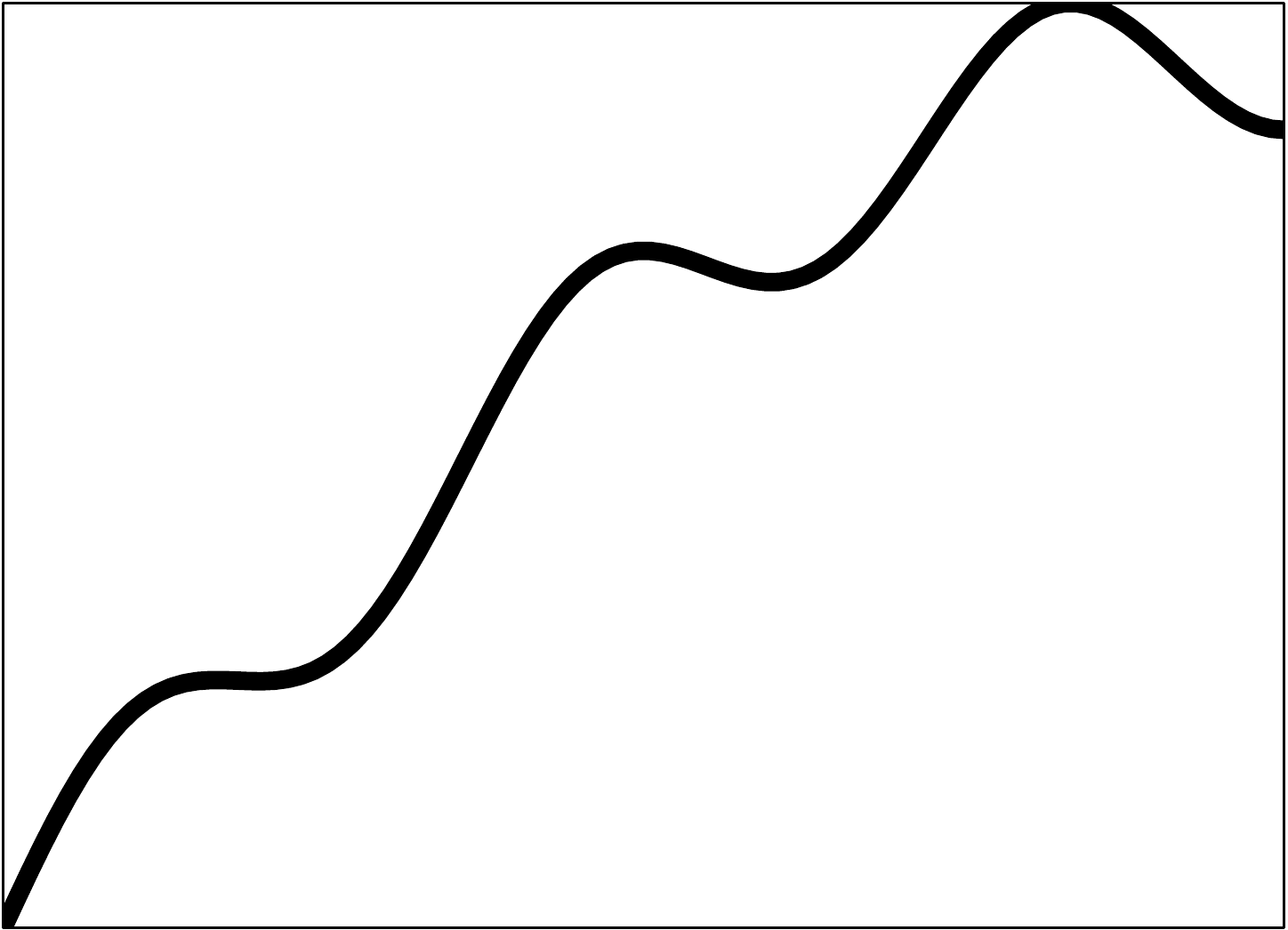}
		\includegraphics[width=\pcgscale\linewidth]{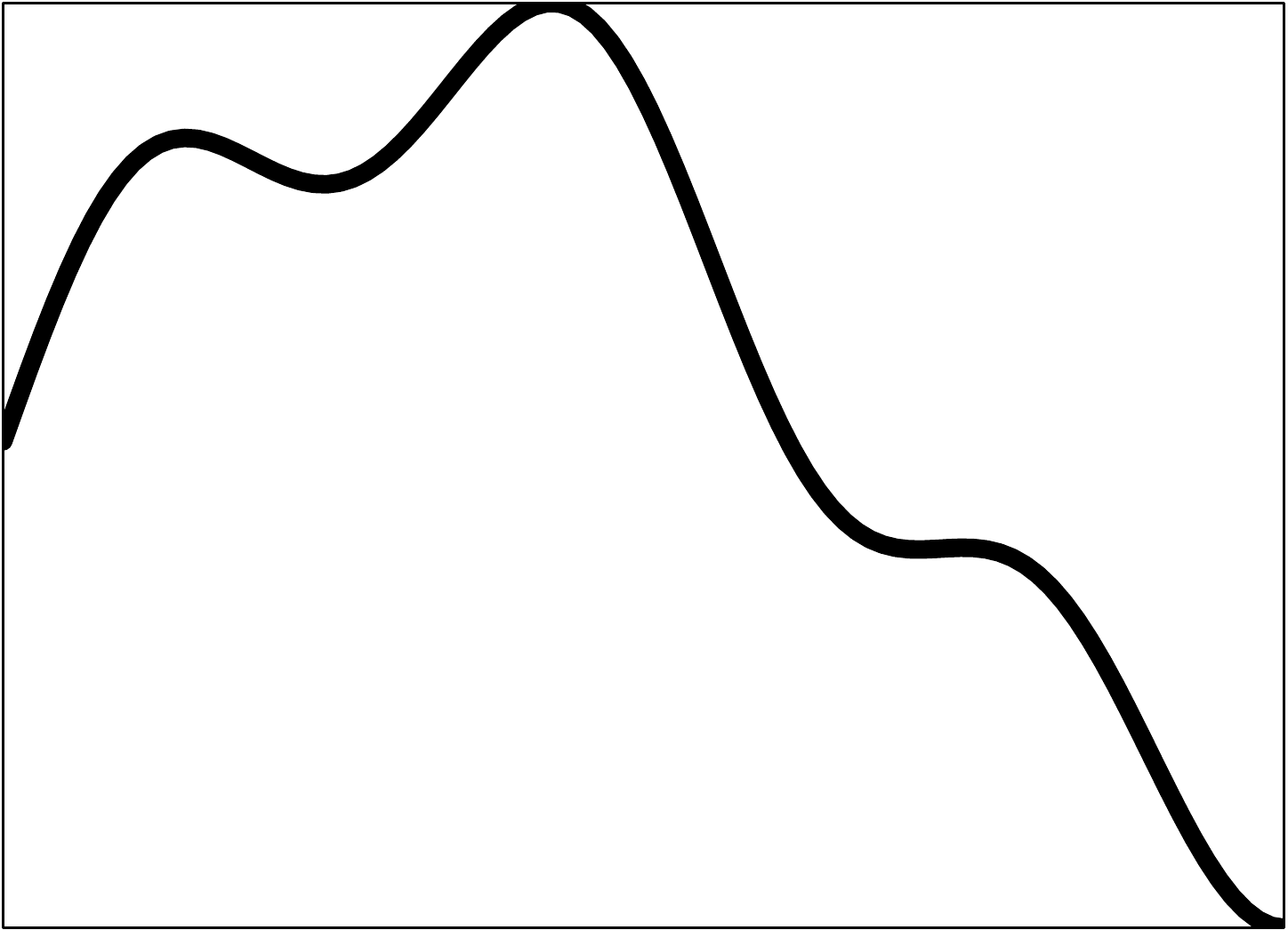}
		\includegraphics[width=\pcgscale\linewidth]{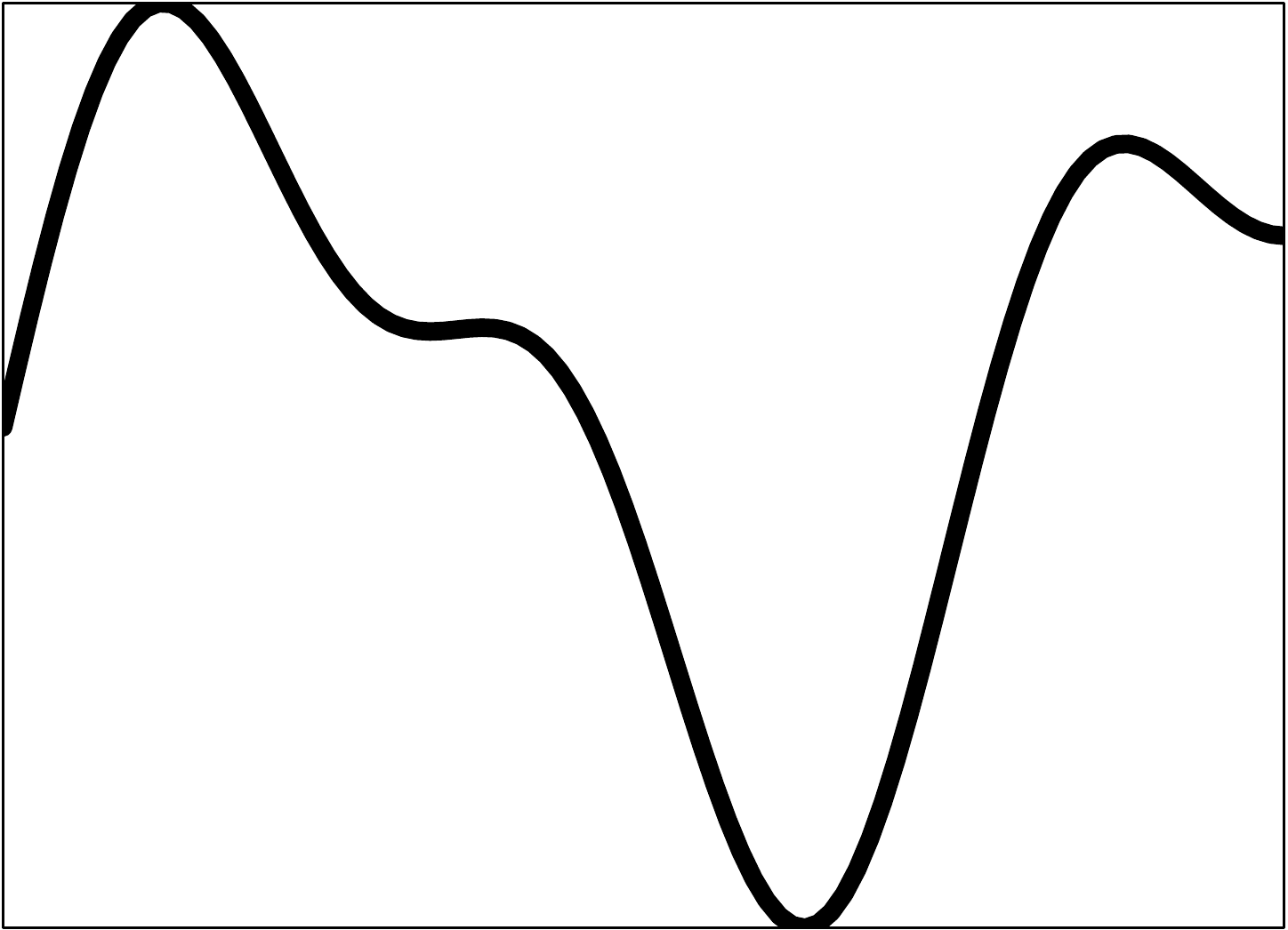}
		\includegraphics[width=\pcgscale\linewidth]{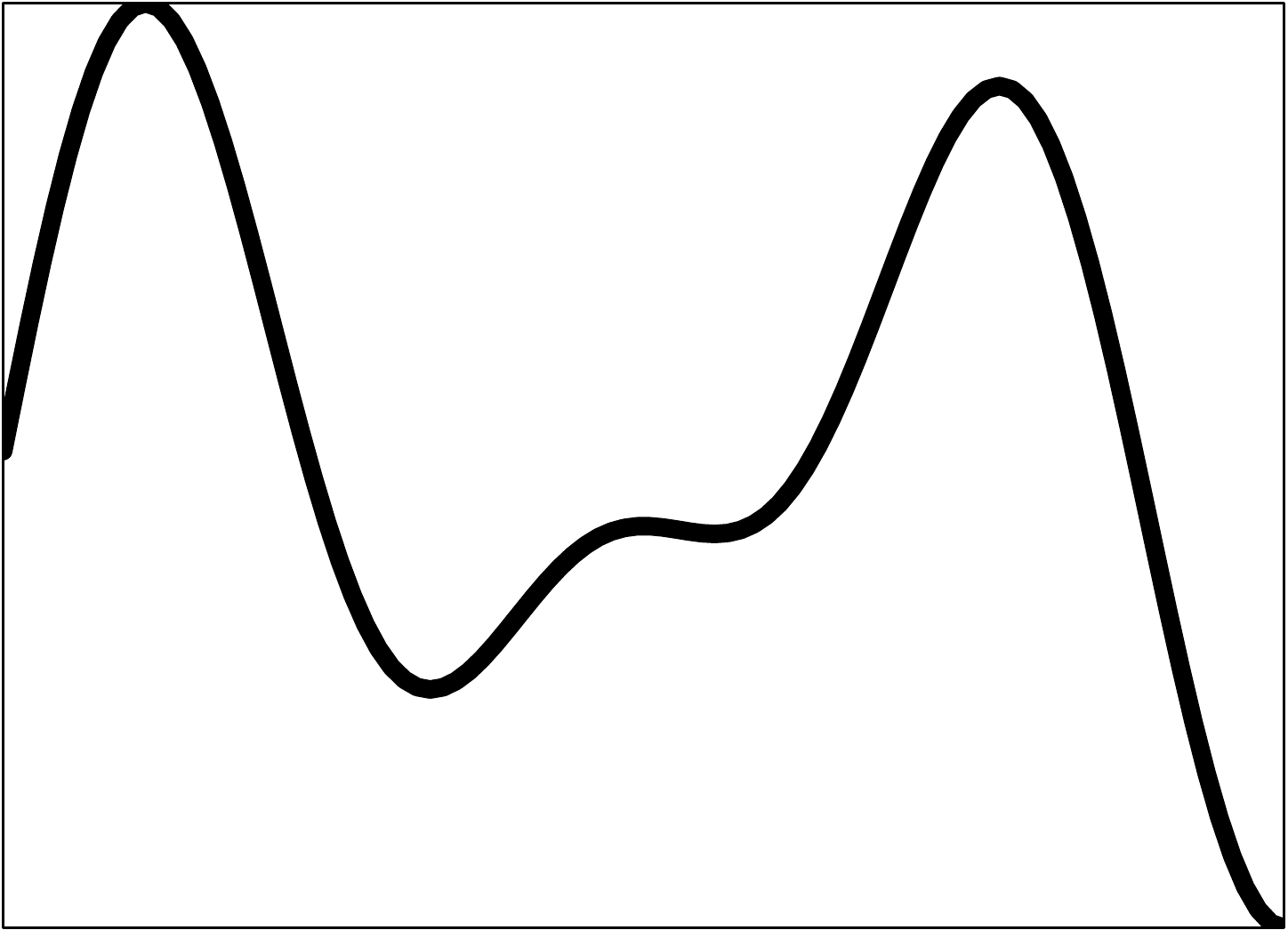}
		\includegraphics[width=\pcgscale\linewidth]{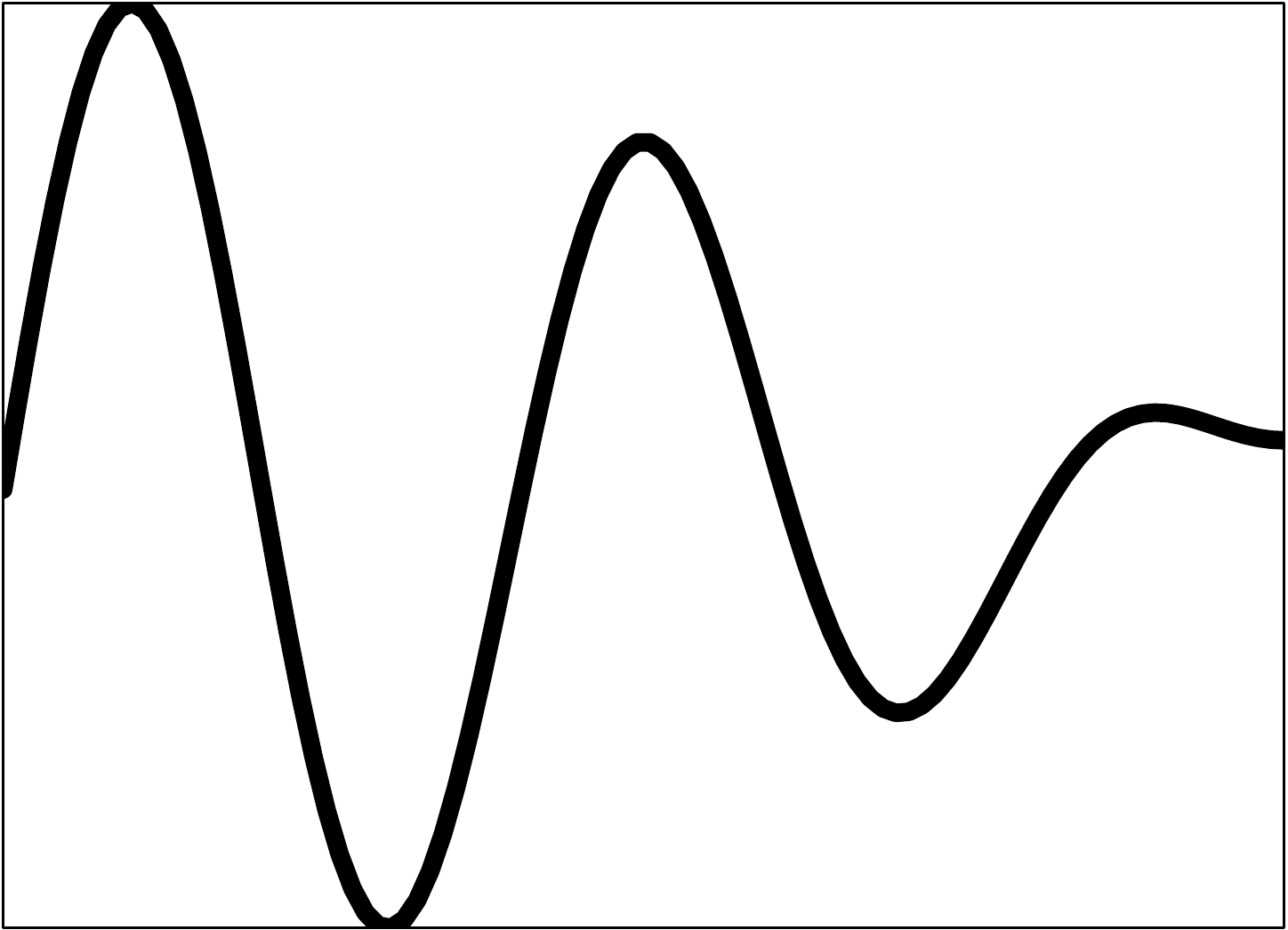}
		\includegraphics[width=\pcgscale\linewidth]{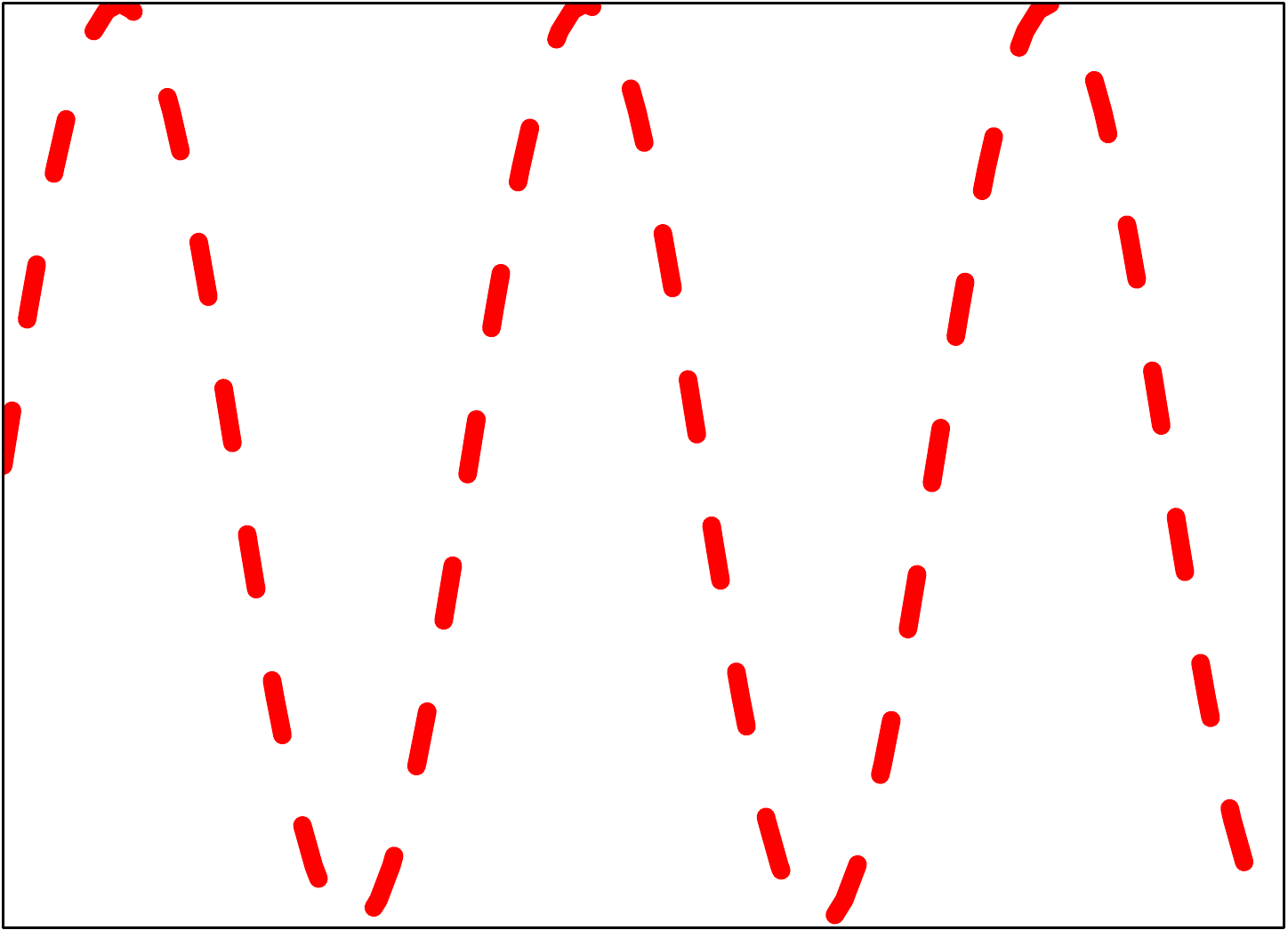}
	\end{minipage}\hfill
	\begin{minipage}{.18\textwidth}
		\centering
		\begin{tikzpicture}[scale=0.6]
		\label{star_poa}
		\def \n {20}
		\def \N {8}
		\def \radius {1.5cm}
		\def \rd {1mm}
		\def \rer {4mm}
		
		\def \margin {8} 
		
		\node[draw, circle, thick,red, dotted] at (360:0mm) (ustar) {};
		\foreach \i [count=\ni from 0] in {0,1,2,3,4,5}{
			\node[draw, circle] at ({\ni*60}:\radius) (u\ni) {};
			\draw (ustar)--(u\ni);
		}
		
		\end{tikzpicture}   
	\end{minipage}
	\caption{Examples of functions in $H^1([0,1])$ that form a 6-nodes star-shaped PCG. Each row corresponds to one star-shaped graph, and in each case the function that acts as the central node is shown with dashed red line. See Example~\ref{exa:deterministic} for the details of the constructions.}
	\label{fig:sobolev:star}
\end{figure}
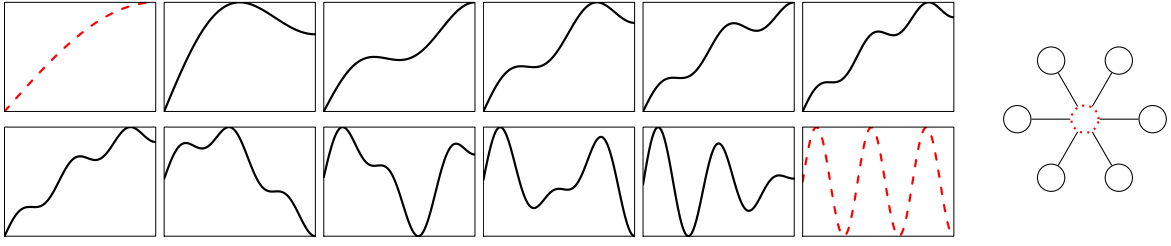

\begin{exa}[PCG for deterministic objects]\label{exa:deterministic}
	The general setup of Definition~\ref{defn:pcg} allows us to define PCGs for purely deterministic objects. For example, let us take $\Hil = \reals^n$ with the usual Euclidean inner product and take $\{\gv_j, j \in [d]\}$ to be a collection of vectors in $\reals^n$. The resulting PCG encodes a form of partial orthogonality among these vectors. For example, let $\{e_1,\dots,e_n\}$ be an orthonormal basis of $\reals^n$ and let $\gv_j = e_1 + e_j$ for $j \in [d]$, assuming $n \ge d$. It is clear that for $j\neq 1$ and any subset $A \subset [d]_j$ containing $1$, $P_{A} x_j = P_1 \gv_j$  and $P_1^\perp \gv_j = e_j$. Hence, for distinct $i,j \neq 1$ and $1 \in A \subset [d]_{ij}$, we have $\gv_i^T P_{A}^\perp \gv_j = \gv_i^T P_1^\perp \gv_j = 0$. In other words, ``conditional'' on any subset $[d]_{ij}$ containing $1$, $i$ and $j$ are orthogonal, i.e., there is no edge  between them in the PCG. One can argue with some more work that all the other edges (i.e., between $1$ and any other node) are present. Hence the PCG is star-shaped with node 1 at the center. The corresponding Gram matrix and its inverse are
	\begin{align}\label{eq:Gram:star}
	\Sigma = 
	\frac1{\gamma - \norm{u}^2} 
	\begin{pmatrix}
	1 & u \\
	u & (\gamma - \norm{u}^2)I +uu^T
	\end{pmatrix}, \quad 
	\Sigma^{-1} = 
	\begin{pmatrix}
	\gamma & -u^T \\
	-u & I
	\end{pmatrix},
	\end{align}
	where $\gamma = d/4$ and $u = \frac12 1_{d-1}$. Here, $1_{d-1}$ is the all ones vector in $\reals^{d-1}$, and we have $\gamma - \norm{u}^2 = 1/4$. The star shape of the PCG is immediately clear from~\eqref{eq:Gram:star} in view of~\eqref{eq:inv:Gram}, or equivalently Lemma~\ref{lem:Gram:char:pair}.
	
	Similarly, one can take $\Hil = \Fc$ for a Hilbert space of functions and let $\{\gv_j\}$ be functions in $\Fc$. For example, take $\Hil = H^1([0,1])$, the Sobolev space of order $1$, defined as the space of absolutely continuous functions $f :[0,1]  \to \reals$ with derivative $f' \in L^2([0,1])$, and $f(0) = 0$. The inner product is $\ip{f,g}_\Hil := \ip{f',g'}_{L^2([0,1])}=\int_0^1 f'(t) g'(t)dt$, and the corresponding norm is a measure of smoothness of the functions. An orthonormal basis for this space is $e_k(t) = \mu_k\sin (t/\mu_k)$ for $k = 1,2,\dots$, where $\mu_k = 2/((2k-1)\pi)$. The exact same construction used earlier, namely $x_j(t) = e_1(t) + e_j(t), \; t \in [0,1],\; j\in [d]$, can be used to obtain a star PCG for the underlying functions, since one gets the exact same Gram matrix as in~\eqref{eq:Gram:star}. In this case, in contrast to the finite-dimensional setting, we can even take $d = \infty$ to have a star PCG with infinitely many leaves. The top row in Figure~\ref{fig:sobolev:star} shows the functions in this construction for $d = 6$. Also shown as the second row is the collection $x_j (t) = e_6(t) + e_j(t), j=1,\dots,6$ which has the same star-shaped PCG with $e_6$ as the central node.\qed
\end{exa}

\begin{exa}[PCG for random vectors]\label{exa:rand:vec}	For a probability measure $\pr$, and $n \in \nats$, let $\Hil = (L^2(\pr))^n$, the $L^2$ space of random vectors $\rv \in \reals^n$ with bounded second moments: $\ex \norm{\rv}_2^2 < \infty$. The inner product here is given by $\ip{X,Y}_\Hil = \ex \ip{X,Y}_{\reals^n} = \sum_{i=1}^n \ex (X_i Y_i)$ where $\ip{\cdot,\cdot}_{\reals^n}$ is the usual Euclidean inner product in $\reals^n$. Fix integers $d$ and $m$, and take a  sequence $\{u_{kj},\; k \in [m], \; j\in [d] \} \subset \reals^n$ of deterministic vectors, a sequence $\{Z^{kj},\; k \in [m], \; j\in [d] \} \subset L^2(\pr)$ of zero-mean (scalar) random variables, and let 
	$X_j = \sum_{k=1}^m Z^{kj} u_{kj} \in \reals^n$ for $ j = 1,\dots,d$.
	Clearly each $X_j \in  (L^2(\pr))^n$ and we have 
	\begin{align*}
	\ip{X_i,X_j}_\Hil = \sum_{k=1}^m \sum_{k'=1}^m \ex[Z^{ki} Z^{k'j}] \ip{u_{ki},u_{k'j}}_{\reals^n}, \quad i,j \in [d]
	\end{align*}
	which can model complex correlations both via the covariance matrices $S_{k,k'} = (\ex[Z^{ki} Z^{k'j}]) \in \reals^{d \times d}$ and the relative positions of $\{u_{kj}\}$ in $\reals^n$. Letting $V_j = \Span\{u_{1j},\dots,u_{mj}\}$, we observe that $X_j \in V_j$, that is $X_j$ is a random element of $V_j$. Thus, we can use this setup to simultaneously model subspace clustering and dimension reduction, especially when assuming $m \ll n$ so that $V_j$s are of much lower dimension than the ambient space, and when some pairs of subspaces are either identical or have large intersections. For example, to see the application to clustering, consider the case where $V_j = V^{(1)}$ for some $j$s and $V_j = V^{(2)}$ for others, with $V^{(1)} \neq V^{(2)}$. Then, the vectors $\{X_j\}$ are naturally divided into two clusters, and their PCG will contain information about the two clusters.
	
	The PCG of Definition~\ref{defn:pcg} for $X = (X_1,\dots,X_d)$ has a random vector, $X_j$, on each node.   The geometry of the subspaces $V_1,\dots,V_d$ affects this PCG of $X$.  For example, when the subspaces are mostly orthogonal, there will be a lot of missing edges in the PCG: If $V_i \perp V_j$ for $i \in A$ and $j \in B$, then,  $\ip{X_i,X_j}_\Hil = 0$ for all $i\in A,\; j \in B$, irrespective of the correlation structure of $Z$. Lemma~\ref{lem:Gram:char:pair} (with $S=[d]_{ij}$), then, implies that there is no edge between $A$ and $B$ in the PCG (since the inverse Gram matrix has a zero block on indices $A\times B$).\qed

\end{exa}

\begin{exa}[PCG for random functions]\label{exa:random:func}
	In the setup of Example~\ref{exa:rand:vec}, we can replace $\reals^n$ with a Hilbert space $\Fc$  of real-valued functions on domain $\Xc$, that is, $\reals^{\Xc}$. We take $\Hil = (L^2(\pr))^{\Xc}$, a space of random functions (or random processes) on domain $\Xc$, with $\ip{F,G}_\Hil = \ex \ip{F,G}_{\Fc}$.  Similar to Example~\ref{exa:rand:vec}, take a deterministic sequence  $\{f_{kj} \in \Fc,\; k\in [m], j\in [d]\}$, (with $m = \infty$ a possibility when $\Fc$ is an infinite-dimensional space) and let $F_j = \sum_{k=1}^m Z^{kj} f_{kj}$. All the discussions of the Example~\ref{exa:rand:vec} go through. The PCG in this case has random functions $\{F_j\}$ as its nodes. Depending on what the inner product of $\Fc$ measures, the absence of an edge could have different meanings. For example, if $\Fc$ is a Sobolev space where the norm measures some form of smoothness of the function, an absence of an edge between $F_i$ and $F_j$ includes some information about the relative smoothness of the pair of functions. 
	
	For example, let $\Fc$ be the Sobolev space of Example~\ref{exa:deterministic} with basis functions $e_1,e_2,\dots$ given there. Consider the star construction with $e_1$ at the center but with random weights: $x_j = Z^{1j} e_1 + Z^{2j} e_j$ for $j\in [d]$. Let $S_{k,k'} = (\ex[Z^{ki} Z^{k'j}]) \in \reals^{d \times d}$ for $k,k' = 1,2$. Note that $S_{2,1} = S_{1,2}^T$. Let $\delta_{i,j} = 1\{i\neq j\}$ be the Kronecker delta function. It is not hard to see that 
	\begin{align*}
	\Sigma_{i,j} := \ip{x_i,x_j} = \big[ S_{1,1} + S_{1,2}^T \delta_{i,1} + S_{1,2}\delta_{1,j} + S_{2,2}\delta_{i,j} \big]_{i,j}, \quad i,j\in[d].
	\end{align*}
	Thus, the off-diagonal elements of the first row of $\Sigma$ ($\Sigma_{1,i},i\neq1$) are determined by the corresponding elements of $S_{1,1} + S_{1,2}^T$, while the off-diagonal elements of $\Sigma$ outside the first row and column are determined by those of $S_{1,1}$ alone. It is clear for example, that elements of $S_{2,2}$ besides $[S_{2,2}]_{1,1}$ have no effect on the Gram matrix. The deterministic case of Example~\ref{exa:deterministic} corresponds to $S_{1,1} = S_{1,2} = $ all-ones $d \times d$ matrix.  The other extreme $S_{1,1} = S_{1,2} = I_d$ leads to the empty PCG. Many other possibilities exist between these two extremes.
	\qed
\end{exa}

We note that Examples~\ref{exa:rand:vec} and~\ref{exa:random:func} provide useful statistical frameworks for modeling multivariate dependencies among random functions or other high-dimensional objects, while retaining some familiar aspects of regression analysis and PCGs. In this sense, they provide avenues for extending classical multivariate statistical analysis to collections of higher-order objects (vectors, matrices, functions, etc.).
\section{Proofs of auxiliary results}

We recall the following notational conventions: For a matrix $\Sigma \in \reals^{d \times d}$, and subsets $A,B \subset [d]$, we use $\Sigma_{A,B}$ for the submatrix on rows and columns indexed by $A$ and $B$, respectively. Single index notation is used for principal submatrices, so that $\Sigma_{A} = \Sigma_{A,\,A}$. For example, $\Sigma_{i,j}$ is the $(i,j)$th element of $\Sigma$ (using the singleton notation), whereas $\Sigma_{ij} = \Sigma_{ij,\,ij}$ is the $2\times 2$ submatrix on $\{i,j\}$ and $\{i,j\}$.

\subsection{Proof of Lemma~\ref{lem:Gram:char:pair}} \label{sec:proof:proj:Sigma}
Let $\Hil$ be the underlying Hilbert space and $x_{i}\in\Hil$ such that $\Sigma = (\ip{x_i,x_j}) \succ 0$. Define the operator  $\Lc_S : \reals^{|S|} \to \Hil$  by $\Lc_S a = \sum_{i \in S} a_i x_i$. The adjoint operator $\Lc_S^* : \Hil \to \reals^{|S|}$ is given by $\Lc^*_S y = (\ip{y,x_i})_{i \in S}$. Since $\Sigma\succ 0$, we have $\Sigma_S = (\ip{x_i,x_j})_{i,j\in S} \succ 0$ for all $S \subset [d]$. By considering $\Sigma_{S}:\reals^{|S|}\to\reals^{|S|}$ as an operator, we have the identification $\Sigma_S = \Lc_S^* \Lc_S$. To simplify notation, let $\Lc_i = \Lc_{\{i\}}$, so that $\Lc_i a = a x_i$ for any $a \in \reals$ and  $\Lc_i^* y = \ip{y,x_i}$ for any $y \in \Hil$. Then, 
\begin{align}\label{eq:PS:in:terms:of:LS}
    P_S = \Lc_S (\Lc_S^* \Lc_S)^{-1} \Lc_S^* = \Lc_S \Sigma_S^{-1} \Lc_S^*.
\end{align} 
(To see~\eqref{eq:PS:in:terms:of:LS}, note that the projection $P_S$ is characterized by the residual $x - P_S x$, for any $x \in \Hil$, being orthogonal to $\ran(P_S) = \ran(\Lc_S)$, that is, $x - P_S x \in [\ran(\Lc_S)]^\perp = \ker(\Lc_S^*)$. Thus, $P_S$ is characterized by $\Lc_S^*(I - P_S) = 0$ which is satisfied by~\eqref{eq:PS:in:terms:of:LS}. Alternatively, write $P_S x = \Lc_S a$ so that $\Lc_S^*( x - \Lc_S a) = 0$ and solve for $a$.) 

\smallskip
It follows that $P_i P_S^\perp P_j = 0$ is equivalent to $\Lc_i \Sigma_i^{-1} \Lc_i^* P_S^\perp \Lc_j \Sigma_j^{-1} \Lc_j^* = 0$. Note that $\Sigma_{i}$ and $\Sigma_j$ are both positive scalars ($i$th and $j$th diagonal entries of $\Sigma$). Hence we can take them out and get the equivalent statement $\Lc_i  \Lc_i^* P_S^\perp \Lc_j \Lc_j^* = 0$. We also note that 
$   \Lc_i  (\Lc_i^* P_S^\perp \Lc_j) \Lc_j^* = (\Lc_i^* P_S^\perp \Lc_j) \Lc_i\Lc_j^*$
since the expression in parentheses is a scalar. $\Lc_i \Lc_j^*$ is not the identically zero operator, hence we get the equivalent statement
\begin{align}
    \begin{split}\label{eq:temp:578}
     0 = \Lc_i^* P_S^{\perp} \Lc_j = \Lc_i^* (I-P_S) \Lc_j  &= 
     \Lc_i^* \Lc_j - \Lc_i^*\Lc_S \Sigma_S^{-1} \Lc_S^*  \Lc_j  \\
     &= \Sigma_{ij} - \Sigma_{S,i}^T \Sigma_S^{-1} \Sigma_{S,j}
    \end{split}
\end{align}
where $\Sigma_{S,i} \in \reals^{|S|}$ is the vector obtained from rows indexed by $S$ in column $i$ of $\Sigma$. By a Schur complement argument, \eqref{eq:temp:578} in turn is equivalent to the determinant of
$
    \big(\begin{smallmatrix}
    \Sigma_{ij} & \Sigma_{S,i}^T \\
    \Sigma_{S,j} & \Sigma_S
    \end{smallmatrix}\big) = \Sigma_{Si,Sj}
$ being zero, as desired.

\subsection{Proof of Proposition~\ref{prop:proj:Sigma:gen}}\label{sec:proof:proj:Sigma:gen}
    Let us start by computing $P_S x_i$ using the notation used in the proof of Lemma~\ref{lem:Gram:char:pair}. Since $x_j = \Lc_j 1$, it is enough to write $P_S \Lc_i = \Lc_S \Sigma_S^{-1} \Lc_S^* \Lc_j = \Lc_S (\Sigma_S^{-1} \Sigma_{S,j})$ (cf. \eqref{eq:PS:in:terms:of:LS}). This means that $\Sigma_S^{-1} \Sigma_{S,j}$ is the coefficient of expansion of $P_A x_j$ in the basis of $\{x_i: i \in S\}$.  That is, $[\beta_j(S)]_S = \Sigma_S^{-1} \Sigma_{S,j}$ in the notation of Definition~\ref{defn:nhbd}. 
    
    We can now obtain an equivalent statement to $P_A P_S^\perp P_B =0$ in terms of the Gram matrix $\Sigma$. The condition is equivalent to $P_A P_S^\perp x_j =0$ for all $j \in B$. This in turn is equivalent to $P_A x_j = P_A P_S x_j $. Since $P_S x_j = \sum_{i \in S} [\beta_j(S)]_i x_i$, we have
    \begin{align*}
        P_A P_S x_j = \sum_{i \in S} [\beta_j(S)]_i (P_A x_i) = 
        \sum_{i \in S} [\beta_j(S)]_i \Big(\sum_{k \in A} [\beta_i(A)]_k x_k\Big).
    \end{align*}
    Using $P_A x_j = \sum_{k \in A} [\beta_j(A)]_k x_k$ and equating the coefficients of $x_k$, we obtain the system of equations given in (b).
    
    To see the equivalence of (b) and (c), let $I_{S,i}$ be the subvector of the identity matrix $I \in \reals^{d \times d}$ on $S \times \{i\}$. Note that we have $[\beta_j(S)]_i = I_{S,i}^T \Sigma_{S}^{-1} \Sigma_{S,j} = \Sigma_{j,S} \Sigma_S^{-1} I_{S,i}$, for all $i \in S$ (in fact true for all $i \in [d]$), where the second equality is by symmetry of $\Sigma$. The RHS of~\eqref{eq:system:of:beta} is equal to 
    \begin{align*}
        \sum_{i \in S} \big(\Sigma_{j,S} \Sigma_S^{-1} I_{S,i}\big) 
        \big(\Sigma_{i,A} \Sigma_A^{-1} I_{A,k} \big) = 
    \Sigma_{j,S} \Sigma_S^{-1} \sum_{i \in S} \big(I_{S,i} 
        \Sigma_{i,A}\big) \Sigma_A^{-1} I_{A,k},
    \end{align*}
    the latter summation being equal to $I_{S,S} \Sigma_{S,A} = \Sigma_{S,A}$. Similarly, the LHS of~\ref{eq:system:of:beta} is equal to $\Sigma_{j,A} \Sigma_A^{-1} I_{A,k}$. Hence,~\ref{eq:system:of:beta} is equivalent to
    \begin{align*}
         \Sigma_{j,A} \Sigma_A^{-1} I_{A,k} =
          \Sigma_{j,S} \Sigma_S^{-1} \Sigma_{S,A} \Sigma_A^{-1} I_{A,k}, \quad \forall k \in A,\; j \in B,
    \end{align*}
    or in matrix form $\Sigma_{B,A} \Sigma_A^{-1} I_{A,A} =
    \Sigma_{B,S} \Sigma_S^{-1} \Sigma_{S,A} \Sigma_A^{-1} I_{A,A}$. Dropping the identity $I_{A,A}$ and rearranging we have $(\Sigma_{B,A}  - \Sigma_{B,S} \Sigma_S^{-1} \Sigma_{S,A}) \Sigma_A^{-1} = 0$ which is equivalent to (c).
    
    Let us now show the equivalence of~\ref{eq:PAPSperpPB} and~\ref{eq:pairwise:beta}. First we note that~\ref{eq:PAPSperpPB} is equivalent to $P_i P_S^\perp P_j =0$ for all $i \in A$ and $j \in B$. Alternatively, it is equivalent to $\ip{x_i, P_S^\perp x_j}_\Hil = 0$ for all $i \in A$ and $j \in B$. Let us fix $i\in A$ and $j \in B$, and show that
    \begin{align*}
        \ip{x_i, P_S^\perp x_j}_\Hil =  0 \iff [\beta_j(Si)]_i = 0
    \end{align*}
    which establishes $\ref{eq:PAPSperpPB}\!\iff\!\ref{eq:pairwise:beta}$. By definition of SEM coefficients in~\eqref{defn:nhbd}, we have $P_{Si}\, x_j = \sum_{k \in Si} [\beta_j(Si)]_k x_k$. To simplify notation, let us write $\alpha =  \beta_j(Si)$ and $u = P_{Si} x_j$. We have
    \begin{align}\label{eq:temp:86969}
        u = \sum_{k \in Si} \alpha_k x_k,\quad x_j = P_{Si}^\perp x_j + u.
    \end{align}
    Applying $P_S^\perp$ to both sides of the first equation, we get $P_S^\perp u = \alpha_i P_S^\perp x_i$. (Since, $P_S^\perp x_k =0,\; \forall k\in S$.) Taking the inner product with $x_i$ (dropping $\Hil$ subscript for simplicity), we obtain
    \begin{align*}
        \alpha_i \ip{x_i,P_S^\perp x_i} &= \ip{x_i,P_S^\perp u} \\
        &= \ip{x_i, P_S^\perp (x_j - P_{Si}^\perp x_j)} \\
        &= \ip{x_i, P_S^\perp x_j} -\ip{x_i, P_{Si}^\perp x_j} &&(\text{Since $P_{Si}^\perp \le P_S^\perp$, i.e. $P_{Si}^\perp   P_S^\perp = P_{Si}^\perp$}) \\
        &= \ip{x_i, P_S^\perp x_j} -\ip{P_{Si}^\perp x_i,  x_j} && (\text{Projections are self-adjoint.})\\
        &= \ip{x_i, P_S^\perp x_j} &&(\text{Since $P_{Si}^\perp x_i=0$.})
    \end{align*}
    Note that $P_{Si}^\perp \le P_S^\perp$ is equivalent to $P_S \le P_{Si}$ which is perhaps easier to see. Since assumption $\Sigma \succ 0$ implies $\ip{x_i,P_S^\perp x_i} > 0$, we have the equivalence of $\alpha_i = 0$ and $\ip{x_i, P_S^\perp x_j} = 0$ which is the desired result. Note that we  have also established~\eqref{eq:explit:pairwise:SEM}.
    The proof is complete.

\subsection{Proof of Lemma~\ref{lem:pairwise:global}}
 \label{sec:proof:pairwise:global}
The idea is to reduce the general case to the case where the underlying random variables are Gaussian, in which case PCG is the same as CIG. The result the follows that for CIGs.

The first step is to show that PCGs are preserved under isometries between Hilbert spaces. Let us recall some facts about isometries: An operator $L \in B(\Hil_1,\Hil_2)$ between two Hilbert spaces is called an isometry if $\vnorm{L \xi}_{\Hil_2} = \vnorm{\xi}_{\Hil_1}$ for all $\xi \in \Hil_1$, i.e. it preserves norms. Equivalently (using polarization identity), it is an isometry if $\ip{L\xi,L\eta}_{\Hil_2} = \ip{\xi,\eta}_{\Hil_1}, \forall \xi,\eta \in \Hil_1$, i.e., it preserves inner products. One has that $L$ is an isometry iff $L^* L = I_{\Hil_1}$. 

\textbf{Step 1} (PCGs are preserved by isometries): Let $\Hil_0 := \Span\{x_1,\dots,x_d\} \subset L^2(\pr)$.
Let $\Kc$ be a $d$-dimensional Hilbert space and let us define the linear operator $V : \Kc \to \Hil_0$ by taking a linearly independent set $\{y_1,\dots,y_d\} \subset \Kc$ with and letting $V y_i = x_i, \forall i \in [d]$. We can further assume that $\{y_i\}$ are chosen such that $\ip{y_i,y_j}_\Kc = \Sigma_{ij} = \ip{x_i,x_j}_\Hil$. It then follows that $V$ is an isometry, hence $V^*V = I_\Kc$. On the other hand, $V$ is also surjective, hence a unitary operator, hence $VV^* = I_{\Hil_0}$.

Let $Q_A = V^* P_A V$. We claim that this is the (orthogonal) projection operator onto the span of $y_A := \{y_i : i \in A\}$.
We have, for $i \in A$,  $Q_A y_i = V^* P_A x_i = V^*x_i = V^* V y_i = y_i$. On the other hand, assume $z \perp_{\Kc} y_i, \forall i \in A$. Then, for all $i \in A$, $ 0 = \ip{z,y_i}_{\Kc} = \ip{Vz,Vy_i}_{\Hil_0} = \ip{Vz,x_i}_{\Hil_0}$. It follows that $P_A V z = 0$, hence $Q_A z=0$.

This shows that any of the pairwise or global $\Hil$-Markov properties are preserved under isometries:
Note that $Q_A^\perp := I_{\Kc} - Q_A = V^*(I_{\Hil_0} - P_A)V = V^* P_A^\perp V$. We conclude that $Q_A Q_S^\perp Q_B = V^* (P_A P_S^\perp P_B) V$ and $P_A P_S^\perp P_B = V (Q_A Q_S^\perp Q_B) V^*$. Hence,
\begin{align}\label{eq:temp:567}
P_A P_S^\perp P_B = 0 \iff Q_A Q_S^\perp Q_B = 0.
\end{align}

\medskip
\textbf{Step 2} (Reduction to Gaussian case): Now take $y = (y_1,\dots,y_d) \sim N(0,\Sigma)$, say on the same probability space, and let $\Kc = \Span\{y_1,\dots,y_d\}$. By~\eqref{eq:temp:567}, and using the assumption that $x$ satisfies pairwise $L^2$ property w.r.t. some $G$, we conclude that $y$ also satisfies pairwise $L^2$ property w.r.t. $G$. Then, due to the equivalence of orthogonality and independence for Gaussian random variables, $y$ satisfies usual pairwise Markov property (defined via conditional independence) w.r.t. $G$. That is, $G$ is a CIG for $y$.

Invoking~\eqref{eq:temp:567} again, it is enough to show that if $A$ and $B$ are separated by $S$, then $Q_A Q_S^\perp Q_B = 0$, which is equivalent to $\ip{u,Q_S^\perp v}_{\Kc} = 0, \forall u \in \Span\{y_A\}, \; v \in \Span\{y_B\}$. Since $S$ separates $A$ and $B$, we know that $u$ and $v$ are independent given $y_S$. This follows from~\cite[Thm.~3.9, p.~35]{Lauritzen1996} for CIGs, since $N(0,\Sigma)$ has a.e. positive density w.r.t. Lebesgue measure, assuming  $\Sigma \succ 0$.  In particular, 
\begin{align}\label{eq:temp:345}
\ex[u v | y_S] = \ex[u|y_S] \,\ex[v|y_S] = (Q_S u)(Q_S v)
\end{align}
where the second equality uses Gaussianty again to write $\ex[u|y_S]$ as the $L^2$ projection to the linear span of $y_S$.
Taking expectation on both sides of~\eqref{eq:temp:345}, we have $\ip{u,v}_\Kc = \ip{Q_S u,Q_S v}_\Kc$ from which it follows that $\ip{Q_S^\perp u,Q_S^\perp v}_\Kc=0$, which in turn is equivalent to $\ip{ u,Q_S^\perp v}_\Kc=0$ since $Q_S^\perp$ is self-adjoint and idempotent.

\subsection{Equivalence of the two directed PCG definitions}\label{sec:dPCG:equiv}
In this appendix, we show that the two definitions of a directed PCG given in~\eqref{eq:dPCG:def1} and~\eqref{eq:dPCG:def2} are equivalent (under the acyclicity assumption). Recall that $\pa_j$ and $N_j$ denote, respectively, the sets of parents and non-descendants of node $j$ in the underlying graph $G$. 

 Assume first that~\eqref{eq:dPCG:def1} holds and take $i,j \in [d]$. Since $G$ is a DAG, one of $i$ and $j$ is a non-descendant of the other, say $i \in N_j$. (If $i$ and $j$ are adjacent, this means that $i$ is a parent of $j$.) Then, any parent of $i$ will be a non-descendant of $j$, that is, $\pa_i \subset N_j$. It follows that $P_k P_{\pa_j}^\perp P_j = 0$ for $k \in \{i\} \cup \pa_i$. In particular, $P_{\pa_i} P_{\pa_j}^\perp P_j = 0$ since $P_{\pa_i} = \bigvee_{k \in \pa_i} P_k$ (see Lemma~\ref{lem:sup:P:Q}), and $P_i P_{\pa_j}^\perp P_j = 0$. Since we have 
 \begin{align}
    \begin{split}\label{eq:dPCG:equiv:identity}
    P_i P_{\pa_j}^\perp P_j 
    &= P_i(P_{\pa_i} + P_{\pa_i}^\perp) P_{\pa_j}^\perp P_j \\
    &= P_i P_{\pa_i} P_{\pa_j}^\perp P_j + P_i P_{\pa_i}^\perp P_{\pa_j}^\perp P_j  \\
    &\stackrel{(c)}{=} P_i P_{\pa_i}^\perp P_{\pa_j}^\perp P_j,
    \end{split}
 \end{align}
 we obtain  $P_i P_{\pa_i}^\perp P_{\pa_j}^\perp P_j = 0$, that is~\eqref{eq:dPCG:def2} holds for $i$ and $j$.

 Now assume that~\eqref{eq:dPCG:def2} holds and take $j \in [d]$ and $i \in N_j$. Note that it is enough to establish identity~\eqref{eq:dPCG:equiv:identity}.  An ancestor of $i$ is any node that has a directed path to $i$; let us call such a path, a path \emph{from the top} to $i$. Let the depth of a node $i$ be the length of the longest path from the top to it. We proceed by the induction on the depth of node $i$. For depth zero, we have $\pa_i = \emptyset$, hence $P_{\pa_i} = 0$ and~\eqref{eq:dPCG:equiv:identity} holds trivially. Now assume that we have the result for all non-descendants of $j$ at depth at most $r-1$, and let $i \in N_j$ have depth $r$. Any $k \in \pa_i$ is at depth at most $r-1$ and belongs to $N_j$, hence by the induction hypothesis $P_k P_{\pa_j}^\perp P_j = 0$ for $ k\in \pa_j$. It follows that $P_{\pa_i} P_{\pa_j}^\perp P_j = 0$, hence~\eqref{eq:dPCG:equiv:identity}(c) holds which then means that \eqref{eq:dPCG:equiv:identity} holds as a whole. The proof is complete.

\section{Abstract lattice theorem}\label{sec:abs:lattice}

In this appendix, we present an interesting generalization of the lattice Theorem~\ref{thm:lattice}. The result and its argument are due to Tristan Bice~\citep{Tristan16}.
Let $A$ be a von Neumann algebra and $P$ be its projection lattice, ordered by $p \le q \iff p = pq \iff p = qp$, the latter equivalence being a result of the self-adjointness of $p$ and $q$. The convention is to use lower case letters for elements of $A$ and $P$. Note that a general element $b \in A$ is a bounded operator on the underlying Hilbert space, i.e., $b \in B(\Hil)$. Similarly $p \in P$ is a projection operator in $B(\Hil)$.
 
\begin{thm}
     $Q = \{q \in P:\, p a = q a\}$ is a complete sublattice of $P$,  $\forall a \in A$ and $p \in P$.
 \end{thm} 
\begin{proof}
    
    Let $[b]$ be the range projection of any $b \in A$, i.e., projection onto the closure of the range of $b$. For any $q \in P$ and $a \in A$, we have $[qa] \le q$ (since the range of $qa$ is included in the range of $q$). Also note the identity (2) $b = [b] b, \; \forall b \in A$.

    If $R \subset Q$, then for all $q \in R$, we have (1) $[pa] = [qa] \le q$, hence $[pa]$ is lower bound on $R$. Letting $r := \bigwedge R$, by definition of infimum, $ [pa] \le r \le q, \; \forall q \in R$, hence $[qa] \le r \le q, \forall q \in R$ by (1).  Hence,
    \begin{align*}
        ra &= r qa & (\text{By} \; r \le q \iff r = rq) \\
        & = r[qa] qa & (\text{By (2) with $b = qa$}) \\
        & = [qa] qa & (\text{By}\; [qa] \le r \iff [qa] = r [qa]) \\
        & = qa & (\text{By (2) with $b = qa$)}\\
        &= pa,
    \end{align*}
    showing that $r \in Q$. So, $Q$ is closed under infima. 

    Since $pa = qa \iff p^\perp a = q^\perp a$  (recall $p^\perp = 1 - p$), it follows that $Q^\perp := \{q^\perp :\; q \in Q\}$ is also closed under infima. But since $q^\perp \le p^\perp \iff p \le q$, it follows that $Q$ is closed under suprema, finishing the proof.
\end{proof}

\end{document}